\pdfoutput=1
\documentclass[
	english,
]{scrartcl}
\usepackage[T1]{fontenc}
\usepackage[utf8]{inputenc}
\usepackage{amsmath,amsfonts,amsthm,amssymb}
\usepackage{xcolor}
\usepackage[colorlinks,citecolor=blue]{hyperref}
\usepackage{preamble}
\usepackage{marginnote}
\usepackage{enumitem}
\setlist{itemsep=0em} % Condense lists
\setlist[enumerate]{label=(\roman*)}

\usepackage[capitalise,nameinlink]{cleveref}
\usepackage{dsfont}
\usepackage{soul,cancel}
\usepackage{graphicx}
\usepackage{mathtools} % For \shortintertext

\newif\ifbiber
\bibertrue

\ifbiber
\usepackage[
	style=authoryear-comp,
	sorting=nyt,
	url=true, 
	doi=true,                % dont Print the DOI.
	eprint=true,
	uniquename=allinit,        % Disambiguate between D and G Wachsmuth
	maxbibnames=99,          % Do not use ``et al'' in the bibliography -- now: only with 99+ authors ;)
	maxcitenames=2,
	uniquelist=minyear,
	dashed=false,            % Do not use ---- to replace the authors in the bibliography (if the previous entry has the same authors)
	sortcites=true,          % Do sort if multiple keys in one \cite{}
	backend=biber,           % Use biber as backend
	safeinputenc,            % For some accents, see http://tex.stackexchange.com/questions/170562
]{biblatex}

\DeclareCiteCommand{\cite}{%
	\ifbibmacroundef{cite:init}{}{\usebibmacro{cite:init}}\usebibmacro{prenote}%
}{%
	\usebibmacro{citeindex}%
	\printtext[bibhyperref]{\usebibmacro{cite}}%
}{%
	\ifbibmacroundef{cite:init}{\multicitedelim}{}%
}{%
	\usebibmacro{postnote}%
}%
\DeclareCiteCommand{\parencite}[\mkbibbrackets]{%
	\ifbibmacroundef{cite:init}{}{\usebibmacro{cite:init}}\usebibmacro{prenote}%
}{%
	\usebibmacro{citeindex}%
	\printtext[bibhyperref]{\usebibmacro{cite}}%
}{%
	\ifbibmacroundef{cite:init}{\multicitedelim}{}%
}{%
	\usebibmacro{postnote}%
}%

\let\cite\parencite

\DeclareNameAlias{sortname}{last-first}
\DeclareNameAlias{default}{last-first}

\else

\usepackage{doi,natbib}

\fi

\usepackage{microtype}

\newcommand\R{\mathbb{R}}

\renewcommand\d{\mathrm{d}}
\newcommand\dx{\d x}

\newcommand{\weakly}{\rightharpoonup}

\newcommand{\anni}{^\perp}

\newcommand\TT{\mathcal{T}}

\newcommand\UU{\mathcal{U}}

\DeclareMathAlphabet{\mathpzc}{OT1}{pzc}{m}{it}
\newcommand\oo{\mathpzc{o}}

\newcommand\closure{\operatorname{cl}}
\newcommand\sgn{\operatorname{sgn}}

\DeclareMathOperator*{\argmin}{arg\,min}

\newtheorem{theorem}{Theorem}[section]
\newtheorem{lemma}[theorem]{Lemma}
\newtheorem{proposition}[theorem]{Proposition}
\newtheorem{assumption}[theorem]{Assumption}

\newtheorem{corollary}[theorem]{Corollary}
\newtheorem{remark}[theorem]{Remark}
\newtheorem{definition}[theorem]{Definition}

\crefname{assumption}{Assumption}{Assumptions}

\definecolor{darkgreen}{rgb}{0,0.5,0}
\definecolor{darkred}{rgb}{0.8,0,0}
\begin{document}
%fakesection: Title and co
\title{On Second-Order Optimality Conditions for Optimal Control Problems Governed by the Obstacle Problem\footnote{
This research was supported by the German Research Foundation (DFG) under grant number WA 3636/4-1
 within the priority program ``Non-smooth and Complementarity-based Distributed Parameter
Systems: Simulation and Hierarchical Optimization'' (SPP 1962). The first author gratefully acknowledges 
the support by the International Research Training Group Munich-Graz, IGDK 1754,
funded by the German Research Foundation (DFG) and the Austrian Science Fund (FWF).}
}

\author{%
 	Constantin Christof%
 	\footnote{%
		Technische Universit\"at M\"unchen,
		Faculty of Mathematics,
		85748 Garching bei M\"unchen, Germany,
		\url{https://www-m17.ma.tum.de/Lehrstuhl/ConstantinChristof},
  		\email{christof@ma.tum.de}%
 		}%
 	\and
 	Gerd Wachsmuth%
 	\footnote{%
		Brandenburgische Technische Universit\"at Cottbus-Senftenberg, 
		Institute of Mathematics, 
		03046 Cott\-bus, Germany, 
 		\url{https://www.b-tu.de/fg-optimale-steuerung},
 		\email{gerd.wachsmuth@b-tu.de}%
 	}%
 }
 % \date{\today}
 % \publishers{}
 %\dedication{}
\maketitle

\begin{abstract}
This paper is concerned with second-order optimality conditions 
for Ti\-kh\-on\-ov regularized optimal control problems governed by the obstacle  problem.
Using a simple observation that allows to characterize
the structure of optimal controls on the active set, 
we derive various conditions that guarantee the
local/global optimality of first-order stationary points and/or the local/global quadratic growth of the reduced objective function.
Our analysis extends and refines existing results from the literature, and 
also covers those situations where the problem at hand
involves additional box-constraints on the control. 
As a byproduct, our approach shows in particular that Tikh\-onov regularized
optimal control problems for the obstacle problem 
can be reformulated as state-constrained optimal control problems for the Poisson equation,
and that problems involving a subharmonic obstacle and a convex objective function are 
uniquely solvable. 
The paper concludes with three counter\-examples which illustrate
that rather peculiar effects can occur in the analysis 
of second-order optimality conditions for optimal control problems governed by the obstacle problem, 
and that necessary second-order conditions
for such problems may be hard to derive.
\end{abstract}

\begin{keywords}
obstacle problem, second-order condition, non-smooth optimization, optimal control, strong stationarity, global optimality, quadratic growth,
control constraints
\end{keywords}

\begin{msc}
\mscLink{35J86},
\mscLink{49J40},
\mscLink{49K21}
 \end{msc}

\section{Introduction}
\label{sec:intro}

The aim of this paper is to study second-order optimality conditions for Ti\-kh\-on\-ov regularized optimal
control problems governed by the classical obstacle problem, i.e., for minimization problems of the form
\begin{equation}
	\label{eq:prob}
	\tag{\textup{P}}
	\begin{aligned}
		\text{Minimize} \quad & J(y, u) := j(y) + \frac{\alpha}{2}\|u\|_{L^2}^2 \\
		\text{w.r.t.}\quad &(y,u) \in H_0^1(\Omega) \times L^2(\Omega) \\
		\text{s.t.}\quad & y \in K, \ \ \left \langle -\Delta y, v - y \right \rangle \geq \left \langle u, v - y \right \rangle \ \ \forall v \in K \\
		\text{and}\quad& u_a \leq u \leq u_b,
	\end{aligned}
\end{equation}
where $K := \{ v \in H_0^1(\Omega) \mid v \geq \psi \text{ a.e.\ in }\Omega \}$.
For the precise assumptions on the quantities $j$, $\alpha$, $\Omega$, $\psi$, etc.\ in \eqref{eq:prob}, we refer to \cref{sec:2}.
The main difficulty in deriving optimality conditions for
problems of the type \eqref{eq:prob} is the non-differentiability of the solution operator $S : u \mapsto y$
associated with the obstacle problem 
\begin{equation*}
y \in K, \ \ \left \langle -\Delta y, v - y \right \rangle \geq \left \langle u, v - y \right \rangle \ \ \forall v \in K
\end{equation*}
which appears as a constraint.
Because of this non-smoothness, standard results and analytical tools are typically inapplicable, and 
one has to work with rather involved stationarity concepts to construct, e.g., conditions which
are sufficient for the local optimality of a given control $\bar u$. 
In the literature, the approach that is most commonly used  in the context of second-order optimality conditions
to overcome the lack of regularity of the 
control-to-state mapping  $S : u \mapsto y$ in \eqref{eq:P} 
is to employ a strong stationarity system in the sense of \cite{Mignot1976} to derive a Taylor-like expansion for the 
map $u \mapsto J(S(u), u)$ and to subsequently analyze the growth behavior of the reduced objective function of \eqref{eq:P}
in the neighborhood of stationary points directly (cf.\ the results in \cref{sec:3}). 
This strategy has been pursued, e.g., in \cite{KunischWachsmuth2012} and 
\cite{AliDeckelnickHinze2018}
and typically gives rise to second-order sufficient optimality conditions which, along with
inequalities involving the second derivative $j''$, also contain assumptions 
on the sign or the size of dual quantities in the vicinity of the contact set. 
For related work on the optimal control of elliptic variational inequalities,
see also \cite{Harder2018,StrongStationarityConstraints2014,BergouniouxMignot2000,
Bergounioux1997,BergouniouxTiba1998,MeyerThoma2013,ItoKunisch2000,HintermullerKopacka2009,Outrata2011}.

In the present paper, we demonstrate that it is possible to improve
the known second-order conditions for problems of the type \eqref{eq:P} 
by exploiting the composite structure of the objective function $J$. To be more precise, 
in what follows, we show that the Ti\-kh\-on\-ov regularization term in $J$
allows to calculate precisely the values of first-order stationary controls of \eqref{eq:P}
on the contact set, and that the resulting formulas can be used 
to weaken the inequality conditions on the adjoint state/the control employed in 
\cite{KunischWachsmuth2012} and \cite{AliDeckelnickHinze2018} considerably.
As a byproduct, our approach shows in particular that problems of the form \eqref{eq:P}
can be reformulated as state-constrained optimal control problems 
for the Poisson equation, and that \eqref{eq:P} admits 
a unique local/global solution $\bar u$ which satisfies a global quadratic growth condition
 when the function $j$ is convex and the
obstacle $\psi$ is subharmonic. 
For the main results of our analysis, see 
\cref{the:Qstructure,th:reduxsubharmonic,th:enhancedSSC,th:SSCmajor} and
\cref{cor:redux2,cor:subharmonicSSC2}.

We conclude this introduction with a brief overview of the structure and the content of the paper:

\cref{sec:2} is concerned with preliminaries. Here, we clarify the notation, discuss 
the existence and properties of solutions of \eqref{eq:P}, and recall some classical results 
on strong and Bouligand stationarity conditions. 

\cref{sec:3} contains a theorem that essentially summarizes and combines the results of 
 \cite{KunischWachsmuth2012} and \cite{AliDeckelnickHinze2018}. 
In contrast to the second-order conditions found in the literature, the analysis 
of this section also covers those cases where \eqref{eq:P} involves 
box-constraints on the control. 

In \cref{sec:4}, we illustrate that the Ti\-kh\-on\-ov regularization term in $J$
indeed imposes a special structure on the first-order stationary points of \eqref{eq:P},
and that \eqref{eq:P} can indeed be reformulated as a state-constrained optimal control problem 
for the Poisson equation. The results of this section are also applicable when \eqref{eq:P} contains an
additional constraint of the form $y \in Y_{ad}$.

\cref{sec:5} addresses the consequences that the findings of  \cref{sec:4} have for the analysis 
of problems \eqref{eq:P} with subharmonic obstacles. Here, we prove in particular the already mentioned 
unique solvability in the case of a subharmonic obstacle and a convex $j$, and also discuss some implications 
for the analysis of state-constrained optimal control problems, cf.\ \cref{cor:subharmonicSSC3}. 

In \cref{sec:6}, we apply the results of \cref{sec:4} to problems \eqref{eq:P} with general obstacles. 
The main result of this section, \cref{th:enhancedSSC}, shows that the assumptions on the adjoint state in the second-order conditions 
of \cite{KunischWachsmuth2012} and \cite{AliDeckelnickHinze2018}, which essentially express that 
the adjoint state should not be ``too'' negative in the vicinity of the contact set, are too pessimistic, 
and that it is indeed sufficient when the adjoint state takes values outside of a bounded interval
whose length depends on the curvature of the obstacle under consideration. 

Lastly, \cref{sec:7} contains three counterexamples which illustrate which effects can prevent a 
strongly stationary point of \eqref{eq:P} from being a local optimum. Here, we will see 
in particular that the conditions on the dual quantities in our second-order conditions 
cannot be dropped without major problems. 

\section{Notation, Problem Setting and Preliminaries}
\label{sec:2}

Before we begin with our analysis, 
let us briefly comment on the notation that we employ in this paper: 
In what follows, 
we use the standard symbols $L^p(\Omega)$, $H_0^k(\Omega)$, $H^k(\Omega)$, $W^{k,p}(\Omega)$ and $C^{k, \alpha}(\Omega)$,
$k \in \mathbb{N}$, $1 \leq p \leq \infty$, $0 < \alpha \leq 1$, for the Lebesgue-, Sobolev- and H\"older spaces 
on a bounded domain $\Omega \subset \R^d$. For the precise definitions 
of these spaces and the associated norms and scalar products 
$\|\cdot\|_{L^p}$, $\|\cdot\|_{H^k}$, $\|\cdot\|_{W^{k,p}}$, $\|\cdot\|_{C^{k, \alpha}}$,
$(\cdot, \cdot)_{L^2}$, and $(\cdot, \cdot)_{H^k}$, we refer to \cite{Adams1975,Attouch2006,Evans2010}.
As usual, we denote the dual of $H_0^1(\Omega)$ by $H^{-1}(\Omega)$
and the dual pairing between elements of $H_0^1(\Omega)$ and $H^{-1}(\Omega)$ by $\left \langle \cdot, \cdot \right \rangle$.
With $\Delta$ and $\closure(\cdot)$, we denote the (distributional) Laplacian
and the topological closure of a set, respectively. 
If we want to emphasize that the closure is taken with respect to a particular norm, 
then we add a suitable subscript and write, e.g., $\closure_{H^1}(\cdot)$. With  $\mathds{1}_A : \Omega \to \{0,1\}$ 
we denote the indicator function of a measurable set $A \subset \Omega$, and with  $\{v * 0\}$, 
$* \in \{=,\neq,<,>,\leq, \geq\}$, $v : \Omega \to \R$, the set $\{ x \in \Omega \mid v(x) * 0\}$. Where appropriate, we 
consider $\{v * 0\}$ to be defined up to sets of measure zero and identify $\mathds{1}_A$ with an element of $L^\infty(\Omega)$.
Given a normed space $(V, \|\cdot\|_V)$, an element $\zeta$ of the topological dual $V^*$, $r>0$,
and a convex, non-empty set $L \subset V$,
we further denote with $B_r^V(v) := \{w \in V \mid \|v - w\|_V\leq r\}$ 
the closed ball of radius $r$ in $V$ centered at $v$, with $\zeta^\perp$ the kernel of $\zeta$, 
and with $\TT_L(v) := \closure_V \left ( \R^+(L - v)\right)$
the tangent cone to $L$ in $V$ at $v$, cf.\ \cite[Section 2.2.4]{BonnansShapiro2000}.
Note that additional symbols etc.\ are introduced in this paper wherever necessary. 
For the sake of readability, this supplementary notation is defined where it first appears in the text.

As already mentioned in the introduction, 
the main goal of this paper is to study second-order optimality conditions for optimal 
control problems of the type
\begin{equation}
	\label{eq:P}
	\tag{\textup{P}}
	\begin{aligned}
		\text{Minimize} \quad & J(y, u) := j(y) + \frac{\alpha}{2}\|u\|_{L^2}^2 \\
		\text{w.r.t.}\quad &(y,u) \in H_0^1(\Omega) \times L^2(\Omega) \\
		\text{s.t.}\quad & y \in K, \ \ \left \langle -\Delta y, v - y \right \rangle \geq \left \langle u, v - y \right \rangle \ \ \forall v \in K \\
		\text{and}\quad& u \in U_{ad} := \left \{ w \in L^2(\Omega) \mid u_a \leq w \leq u_b \text{ a.e.\ in } \Omega \right \}.
	\end{aligned}
\end{equation}
Our standing assumptions on the quantities in \eqref{eq:P} are as follows:

\begin{assumption}[Standing Assumptions for the Study of Problem {\eqref{eq:P}}]~
\label{assumption:standing}
\begin{itemize}
\item $d \in \{1,2,3\}$,
\item 
$\Omega \subset \mathbb{R}^d$ is a bounded domain 
that is convex or possesses a $C^{1,1}$-boundary, 
\item $j : H_0^1(\Omega) \to \R$ is twice continuously Fréchet differentiable and bounded from below,
\item $\alpha > 0$ is a given Tikh\-onov parameter,
\item $K := \{ v \in H_0^1(\Omega) \mid v \geq \psi \text{ a.e.\ in }\Omega \}$ with obstacle
 $\psi \in H^2(\Omega)$ such that $K \neq \emptyset$,
\item $u_a, u_b : \Omega \to [-\infty, \infty]$ are measurable functions with
$u_a \leq 0 \leq u_b$ a.e.\ in $\Omega$.
\end{itemize}
\end{assumption}

We remark that the subsequent analysis can be extended straightforwardly to those 
cases where the set $K$ in \eqref{eq:P} 
is of the form $ \{ v \in H_0^1(\Omega) \mid  {\psi_1\leq v\leq \psi_2} \text{ a.e.\ in }\Omega \}$
with functions $\psi_1, \psi_2 \in H^2(\Omega)$ satisfying $\psi_1 \leq \psi_2 - \varepsilon$ a.e.\ in $\Omega$ 
for some $\varepsilon > 0$. We restrict our attention
to the prototypical setting in \cref{assumption:standing} for the sake of simplicity
and to reduce the notational overhead. 

For the analysis of the optimal control problem \eqref{eq:P}, we 
need several known results on the properties of the solution map 
associated with the obstacle problem. We collect these in:

\begin{theorem}[Properties of the Control-to-State Map]
\label{th:solutionmapproperties}
For every $u \in L^2(\Omega)$, there exists one and only one solution $y \in H_0^1(\Omega)$ of the obstacle problem 
\begin{equation}
\label{eq:obstacleproblem}
 y \in K, \qquad \left \langle -\Delta y, v - y \right \rangle \geq \left \langle u, v - y \right \rangle \qquad \forall v \in K.
\end{equation}
This solution satisfies $y \in H_0^1(\Omega) \cap H^2(\Omega)$ and $-\Delta y = u + \lambda $ with 
a unique multiplier 
$\lambda \in L^2(\Omega)$ such that
\begin{equation}
\label{eq:multiplierformula}
\begin{aligned}
0 \leq \lambda
&=
\begin{cases}
-\Delta \psi - u  & \text{ a.e.\ in } \{y = \psi\},
\\
0 &\text{ a.e.\ in } \{ y > \psi \},
\end{cases}
\end{aligned}
\end{equation}
and there exists a constant $C>0$ independent of $u$ with 
\begin{equation}
\label{eq:H2reg}
\|y\|_{H^2} \leq C \left ( \|u\|_{L^2} + \|\Delta \psi\|_{L^2}\right ).
\end{equation}
Further, the solution map $S : u \mapsto y$ is
 globally Lipschitz continuous as a function from $L^2(\Omega)$ to $L^\infty(\Omega)$ and
globally Lipschitz continuous and directionally differentiable as a function from $H^{-1}(\Omega)$ to 
$H_0^1(\Omega)$, and the directional derivative 
$\delta_h := S'(u; h) \in H_0^1(\Omega)$ of $S$ in a point  $u \in L^2(\Omega)$ in a direction $h \in L^2(\Omega)$ is uniquely characterized by the variational inequality
\begin{equation}
\label{eq:VIdirdiff}
\delta_h \in \TT_K(y) \cap \lambda^\perp,\qquad 
  \left \langle  - \Delta \delta_h, z - \delta_h \right \rangle \geq \left \langle h, z - \delta_h \right \rangle
\qquad \forall z \in  \TT_K(y) \cap \lambda^\perp
\end{equation}
with $y := S(u)$, $\TT_K(y) := \closure_{H^1}(\R^+(K - y))$ and $\lambda := -\Delta y - u \in L^2(\Omega)$. 
Here, $\lambda^\perp$ denotes the kernel of $\lambda$ as an element of the dual space $H^{-1}(\Omega)$. 
\end{theorem}

\begin{proof}
The existence of a unique solution $y \in H_0^1(\Omega)$ of \eqref{eq:obstacleproblem}
and the Lipschitz continuity properties of the solution map $S : u \mapsto y$ 
follow from standard results, see \cite[Theorem II-2.1]{KinderlehrerStampacchia1980}  and \cite[Lemma~2.2]{KunischWachsmuth2012},
and the directional differentiability of $S$ and the variational inequality \eqref{eq:VIdirdiff} 
are direct consequences of the polyhedricity of the admissible set $K$ and classical results of Mignot, see \cite{Wachsmuth2016:2,Mignot1976,Christof2018Phd}. 
To establish the $H^2$-regularity of the solution $y$ and the estimate \eqref{eq:H2reg},
one can use exactly the same arguments as in \cite[Chapter IV, Section 2]{KinderlehrerStampacchia1980},
cf.\ \cite[Theorem 9.15, Lemma 9.17]{GilbargTrudinger2001}, \cite[Theorem 3.2.1.2]{Grisvard1985}.

It remains to prove the formula \eqref{eq:multiplierformula} for $\lambda$. 
To this end, we first note that the $H^2$-regularity of the solution $y$, the variational inequality \eqref{eq:obstacleproblem}
and the structure of $K$ imply that $\lambda := -\Delta y - u$ is a non-negative element of $L^2(\Omega)$ which 
vanishes 
a.e.\ in $\{ y > \psi\}$.
From
the lemma of Stampacchia, see \cite[Proposition~5.8.2]{Attouch2006},
we may further deduce that 
$\nabla (y - \psi) = 0$ holds a.e.\ on $\{y = \psi\}$
and that
$\Delta (y - \psi) = 0$ holds a.e.\ on $\{\nabla y = \nabla \psi\}$.
As a consequence, 
$\Delta (y - \psi) = 0$ a.e.\ on $\{y = \psi\}$.
The formula  \eqref{eq:multiplierformula} now follows immediately. 
\end{proof}

The next result about the continuity of $S$ into higher-order Sobolev spaces
seems to be less known.
It can be found in \cite[Theorem 5.4.3]{Rodrigues1987} for the case $\psi = 0$
and in \cite{SchielaWachsmuth2013} for a regularized version
of the obstacle problem.
For convenience, we give its proof.
\begin{theorem}[Lipschitz Estimate for Higher Derivatives]
	\label{thm:hoelder}
	For all $u_1, u_2 \in L^2(\Omega)$ with associated states $y_1 := S(u_1)$, $y_2 := S(u_2)$, it holds
	\begin{equation}
	\label{eq:randomeq82333442}
		% \norm{ y_2 - y_1 }_{H^2(\Omega)} \le C \, \norm{ u_2 - u_1 }_{L^2(\Omega)}^{1/2}
		% \quad\text{and}\quad
		\norm{ \Delta(y_1 - y_2) }_{L^1} \le 2 \, \norm{ u_1 - u_2}_{L^1}.
	\end{equation}
	% holds
\end{theorem}
\begin{proof}
	We have
	\begin{equation}
	\label{eq:ranomeq82362728}
		-\Delta (y_2 - y_1)
		=
		u_2 - u_1 + \lambda_2 - \lambda_1
		,
	\end{equation}
	where $\lambda_i := -\Delta y_i - u_i$, $i = 1,2$.
	To derive \eqref{eq:randomeq82333442} from \eqref{eq:ranomeq82362728}, we proceed as in 
	\cite[Proof of Theorem~5.1]{ItoKunisch2000} and define
	\begin{equation*}
		\rho_\varepsilon(x)
		:=
  \max \left ( -1, \min \left (1, \frac{x}{\varepsilon} \right ) \right )
 		=
		\begin{cases}
			-1 & \text{for } x \le -\varepsilon, \\
			\frac x\varepsilon & \text{for } x \in (- \varepsilon,\varepsilon), \\
			1 & \text{for } x \ge \varepsilon. \\
		\end{cases}
	\end{equation*}
	Since $\rho_\varepsilon$ is continuous and piecewise affine, 
	Stampacchia's lemma yields
	$\rho_\varepsilon(y_2 - y_1) \in H_0^1(\Omega)$
	and 
	$\nabla \rho_\varepsilon(y_2 - y_1) = \rho_\varepsilon'(y_2 - y_1) \, \nabla (y_2 - y_1)$.
	In particular, we may choose $\rho_\varepsilon(y_2 - y_1)$ as a test function in \eqref{eq:ranomeq82362728} to obtain
	\begin{equation}
	\label{eq:randomeq183904738}
	\begin{aligned}
		0
		&\le
		\int_\Omega \rho_\varepsilon'(y_2 - y_1) \, |\nabla (y_2 - y_1)|^2
		\, \dx
		=
		\int_\Omega \nabla \big ( \rho_\varepsilon (y_2 - y_1) \big )\, \nabla (y_2 - y_1)
		\, \dx
		\\
		&\le
		\norm{u_2 - u_1}_{L^1(\{y_2 \ne y_1\})}
		+
		\int_\Omega \rho_\varepsilon(y_2 - y_1) \, (\lambda_2 - \lambda_1)\, \dx
		.
	\end{aligned}
	\end{equation}
	Here, we have exploited that $\rho_\varepsilon' \ge 0$ and $\abs{\rho_\varepsilon} \le 1$.
	Using the dominated convergence theorem,
	we can pass to the limit $\varepsilon \searrow 0$ in \eqref{eq:randomeq183904738}. This yields
	\begin{equation}
	\label{eq:randomeq1735365}
		-\int_\Omega \sgn(y_2 - y_1) \, (\lambda_2 - \lambda_1)\, \dx
		\le
		\norm{u_2 - u_1}_{L^1(\{y_2 \ne y_1\})}
		.
	\end{equation}
	Note that, for almost all $x \in \Omega$, we have 
	\begin{equation}
	\label{eq:randomimps22}
	\begin{aligned}
		y_2(x) = y_1(x) = \psi(x) & \qquad\Rightarrow \qquad \lambda_2(x) - \lambda_1(x) = u_1(x) - u_2(x),
		\\
		y_2(x) = y_1(x) > \psi(x) & \qquad\Rightarrow \qquad \lambda_2(x) - \lambda_1(x) = 0,
		\\
		y_2(x) > y_1(x) & \qquad\Rightarrow\qquad 0 = \lambda_2(x) \le \lambda_1(x),
		\\
		y_1(x) > y_2(x) & \qquad\Rightarrow\qquad 0 = \lambda_1(x) \le \lambda_2(x)
		.
	\end{aligned}
	\end{equation}
	From \eqref{eq:randomeq1735365} and the last two implications in \eqref{eq:randomimps22}, we obtain
	\begin{equation*}
		-\int_\Omega \sgn(y_2 - y_1) \, (\lambda_2 - \lambda_1)\, \dx
		=
		\int_{\set{y_2 \ne y_1}} \abs{\lambda_2 - \lambda_1}\, \dx
		\le
		\norm{u_2 - u_1}_{L^1(\{y_2 \ne y_1\})}
	.
	\end{equation*}
	Further, the first two implications in \eqref{eq:randomimps22} yield 
	\begin{equation*}
		\int_{\set{y_2 = y_1}} \abs{\lambda_2 - \lambda_1}\, \dx
		\le
		\norm{u_2 - u_1}_{L^1(\{y_2 = y_1\})}
		.
	\end{equation*}
	Putting everything together now gives
	\begin{equation*}
		\norm{\lambda_2 - \lambda_1}_{L^1 }
		\le
		\norm{u_2 - u_1}_{L^1 }
	\end{equation*}
	and, as a consequence, 
	\begin{equation*}
		\norm{\Delta(y_1 - y_2)}_{L^1 }
		\le
		2 \, \norm{u_1 - u_2}_{L^1 }
		.
	\end{equation*}
This proves the claim. 
\end{proof}

For the construction of our counterexamples,
we also need the following well-known comparison principle.

\begin{lemma}[Comparison Principle]
	\label{lem:comparison}
	Let $u_1,u_2 \in L^2(\Omega)$ be given such that $u_1 \le u_2$ holds a.e.\ in $\Omega$.
	Then, it also holds $S(u_1) \le S(u_2)$ a.e.\ in $\Omega$.
\end{lemma}
\begin{proof}
	We have to show that $\theta := \max\paren[\big]{S(u_1) - S(u_2), 0} = 0$.
	Clearly, $S(u_1) - \theta \ge \psi$.
	Hence, we can test the VI \eqref{eq:obstacleproblem} for $u = u_1$
	with $v = S(u_1) - \theta$
	and for $u = u_2$ with $v = S(u_2) + \theta$.
	Subtraction of the resulting inequalities and an application of
	Stampacchia's lemma yield
	\begin{equation*}
		-\int_\Omega \abs{\nabla \theta}^2 \, \dx
		\ge
		\scalarprod{u_2 - u_1}{\theta}_{L^2}
		\ge
		0.
	\end{equation*}
	Thus, $\theta = 0$ and the proof is complete. 
\end{proof}

From the properties of the solution operator $S$ and the conditions in \cref{assumption:standing}, 
we immediately obtain
the following two results:

\begin{corollary}[Existence of Optimal Controls]
\label{cor:optexistence}
The optimal control problem \eqref{eq:P}  admits at least one global 
solution $\bar u \in U_{ad}$. 
\end{corollary}

\begin{proof}
The assertion follows straightforwardly from the direct method of calculus of variations, 
the boundedness from below of the function $j$, the continuity of $S$ as a function from $H^{-1}(\Omega)$ to $H_0^1(\Omega)$,
and the compactness of the embedding $L^2(\Omega) \hookrightarrow H^{-1}(\Omega)$. 
\end{proof}
 
\begin{corollary}[Bouligand Stationarity Condition]
\label{cor:BouligandStationarityCondition1}
Every local solution $\bar u$ of the optimal control problem \eqref{eq:P} with associated state $\bar y := S(\bar u)$ satisfies
\begin{equation}
\label{eq:Bouligandstatcond}
\left \langle j'(\bar y), S'(\bar u; h) \right \rangle + \alpha \left ( \bar u, h \right )_{L^2} \geq 0 \qquad 
\forall h \in 
\TT_{U_{ad}}(\bar u).
\end{equation}
Here, $\TT_{U_{ad}}(\bar u)$ denotes the tangent cone to $U_{ad}$ at $\bar u$, i.e.,
$\TT_{U_{ad}}(\bar u) := \closure_{L^2}\left (\R^+\left (U_{ad} - \bar u \right ) \right )$. 
\end{corollary}

\begin{proof}
The claim follows immediately from the directional differentiability and Lipschitz continuity
of the solution map $S$ as a function from $H^{-1}(\Omega)$ to $H_0^1(\Omega)$,
the local optimality of $\bar u$,
and the chain rule, see \cite[Proposition 2.47]{BonnansShapiro2000}. 
\end{proof}

Although very natural, the Bouligand stationarity condition \eqref{eq:Bouligandstatcond} 
is typically of little use in practical applications. A more convenient stationarity 
concept is the following: 

\begin{definition}[Strong Stationarity Condition]
A point  $\bar u \in U_{ad}$ 
with associated state $\bar y \in H_0^1(\Omega)$ and multiplier $\bar \lambda := -\Delta \bar y - \bar u \in L^2(\Omega)$
is called strongly stationary for \eqref{eq:P}
if there exists a triple $(\bar p, \bar \nu, \bar \eta) \in H_0^1(\Omega) \times L^2(\Omega) \times H^{-1}(\Omega)$
such that
\begin{subequations}
\label{eq:strongstationarity}
\begin{align}
\label{eq:strongstationarity_1}
-\Delta \bar p + \bar \eta - j'(\bar y)& = 0 \text{ in } H^{-1}(\Omega),
\\
\label{eq:strongstationarity_2}
\alpha \bar u + \bar p - \bar \nu &=0 \text{ in } L^2(\Omega),
\\
\label{eq:strongstationarity_3}
\bar p &\in \TT_K(\bar y) \cap \bar \lambda^\perp,
\\
\label{eq:strongstationarity_4}
\left \langle \bar \eta, z \right \rangle  &\geq 0\quad 
\forall z \in \TT_K(\bar y) \cap \bar \lambda^\perp,
\\
\label{eq:strongstationarity_5}
\left ( \bar \nu, h \right  )_{L^2}  &\geq 0\quad 
\forall h \in \TT_{U_{ad}}(\bar u).
\end{align}
\end{subequations}
Here,  $\TT_{U_{ad}}(\bar u) := \closure_{L^2}\left (\R^+\left (U_{ad} - \bar u \right ) \right ) $ 
and $\TT_K(\bar y) := \closure_{H^1}(\R^+(K - \bar y))$ again denote the tangent cones to 
$U_{ad}$ and $K$ at $\bar u$ and $\bar y$, respectively, and $\bar \lambda^\perp$ is the kernel of $\bar \lambda$.
\end{definition}

Note that, for every strongly stationary point $\bar u \in U_{ad}$, 
the system \eqref{eq:strongstationarity}, the variational inequality \eqref{eq:VIdirdiff}
and the fact that $\TT_K(\bar y) \cap \bar \lambda^\perp$ is a convex cone imply 
\begin{equation}
\label{eqrandomeq82636}
\begin{aligned}
&\left \langle j'(\bar y), S'(\bar u; h) \right \rangle + \alpha \left ( \bar u, h \right )_{L^2} 
\\
&\qquad = \left \langle -\Delta \bar p + \bar \eta, S'(\bar u; h) \right \rangle + \left ( -\bar p + \bar \nu, h \right )_{L^2} 
\\
&\qquad =  \left \langle -\Delta  S'(\bar u; h) - h, S'(\bar u; h) + \bar p - S'(\bar u; h)\right \rangle 
+ \left \langle \bar \eta, S'(\bar u; h) \right \rangle
+ \left (\bar \nu, h \right )_{L^2} 
\\
&\qquad \geq 0\qquad \forall h \in \TT_{U_{ad}}(\bar u).
\end{aligned}
\end{equation}
Strongly stationary points are thus always Bouligand stationary in the sense of 
\eqref{eq:Bouligandstatcond}.
We would like to emphasize that the converse of this implication, i.e., 
\eqref{eq:Bouligandstatcond} $\Rightarrow$ \eqref{eq:strongstationarity}, 
does not hold in general. See, e.g., \cite[Section 6]{StrongStationarityConstraints2014} 
for two counterexamples.  
However, under mild assumptions on the data, it is possible to prove that 
\eqref{eq:strongstationarity} is indeed a necessary optimality condition for \eqref{eq:P}.
More precisely, the following can be established:

\begin{theorem}[Strong Stationarity as a Necessary Optimality Condition]
\label{th:strongstationaritynecessary}
Suppose that $\bar u$ is a local solution of \eqref{eq:P}
with associated state $\bar y := S(\bar u)$ and multiplier $\bar \lambda := -\Delta \bar y - \bar u$.
Assume further that one of the following conditions is satisfied:
\begin{enumerate}
\item $u_a = -\infty$ and $u_b  = \infty$. \label{item:strong-i}
% \item $\bar u \in H_0^1(\Omega)$, 
% $u_a, u_b \in H^1(\Omega)$ and $u_a \leq 0 < u_b$ quasi-everywhere  in $\Omega$.\label{item:strong-ii}
\item 
$u_a, u_b \in H^1(\Omega)$
 and $u_a \leq 0 < u_b$ quasi-everywhere in $\Omega$. \label{item:strong-iii}
\end{enumerate}
Then, there exists a triple $(\bar p, \bar \nu, \bar \eta) \in H_0^1(\Omega) \times L^2(\Omega) \times H^{-1}(\Omega)$
such that $\bar u$, $\bar y$, $\bar \lambda$, $\bar p$, $\bar \nu$, and $\bar \eta$ satisfy
the strong stationarity system \eqref{eq:strongstationarity}. 
\end{theorem}

Here and in what follows, with quasi-everywhere (q.e.), 
we mean pointwise everywhere up to
sets of $H_0^1(\Omega)$-capacity zero. 
Note that $H_0^1(\Omega)$-q.e.\ in $\Omega$ implies $H^1(\Omega)$-q.e.\ in $\Omega$ and vice versa; see
\cite[Corollary 6.2]{ChristofMuller2018}.
We may thus indeed  write ``$u_a \leq 0 < u_b$ q.e.\ in $\Omega$''
for $u_a, u_b \in H^1(\Omega)$
without any danger of confusion. 
For more details on this topic and the involved concepts, 
we refer to \cite{ChristofMuller2018,Harder2018,BonnansShapiro2000}.

\begin{proof}[Proof of \cref{th:strongstationaritynecessary}]
In case \ref{item:strong-i}, the existence of a triple $(\bar p, \bar \nu, \bar \eta)$
with  \eqref{eq:strongstationarity} follows
from classical results of Mignot, see
\cite[Proposition~4.1]{Mignot1976} and also \cite{Harder2018}, 
\cite[Corollary 6.1.11]{Christof2018Phd}. 
% In case \ref{item:strong-ii},  \eqref{eq:strongstationarity} is an immediate 
% consequence of  \cite[Theorem 5.2, Lemma~5.3]{StrongStationarityConstraints2014}. 
It remains to prove the necessity of the strong stationarity 
system in case \ref{item:strong-iii}.
So let us assume that \ref{item:strong-iii} holds and that $\bar u$ is locally optimal for \eqref{eq:P}. 
Then, 
\cite[Lemma~4.4]{Wachsmuth2014:2} 
yields
 that the tuple $(\bar y, \bar u)$
 is weakly stationary for \eqref{eq:P}.
 In particular, there exists a function $\bar p \in H_0^1(\Omega)$
with
\begin{equation*}
\left ( \alpha \bar u + \bar p, u - \bar u \right )_{L^2} \geq 0 \quad \forall u \in U_{ad}
,\qquad\text{i.e.,}\qquad
\bar u =\max \left (u_a, \min\left (u_b, -\frac{1}{\alpha}\bar p\right ) \right )
.
\end{equation*}
Due to Stampacchia's lemma and the $H^1$-regularity of $u_a$ and $u_b$,
the above implies $\bar u \in H_0^1(\Omega)$. 
Using this regularity and \cite[Theorem 5.2, Lemma 5.3]{StrongStationarityConstraints2014},
the necessity of \eqref{eq:strongstationarity} in case \ref{item:strong-iii} follows immediately. This completes the proof. 
\end{proof}

In the remainder of this paper, we will often simply assume 
that a strongly stationary point $\bar u$ is given. 
The reader should keep in mind that,
by \cref{cor:optexistence,th:strongstationaritynecessary}, 
the existence of such a point and the necessity of the system \eqref{eq:strongstationarity} can be guaranteed under
comparatively mild additional assumptions on the bounds $u_a$ and $u_b$ in \eqref{eq:P}.

\section{SSC Involving Compatibility Conditions}
\label{sec:3}

Having established the existence of optimal controls and the stationarity conditions 
\eqref{eq:Bouligandstatcond} and \eqref{eq:strongstationarity}, 
we now turn our attention to second-order sufficient optimality conditions (SSC) for the problem \eqref{eq:P}. 
To the authors' knowledge, the only contributions that provide such conditions 
so far are \cite[Théorème~4.1]{Mignot1976}, \cite{KunischWachsmuth2012} and \cite{AliDeckelnickHinze2018},
where optimal control problems governed by the obstacle problem 
without control constraints are considered.
In these three papers, the basic idea of the analysis is to employ 
a Taylor-like expansion of the reduced objective function $J(S(u), u)$ 
and certain compatibility assumptions on the sign, the size, or the growth of the multipliers $\bar p$ and $\bar \eta$
in relation to the primal quantities $\bar y $ and $\bar \lambda$
to establish conditions that are sufficient for local or global optimality. 
In the situation of problem \eqref{eq:P}, we can use the system \eqref{eq:strongstationarity}
to obtain a similar expansion of the reduced objective function as the following lemma shows.

\begin{lemma}
\label{lemma:Taylorlikeexpansion}
Suppose that $\bar u \in U_{ad}$ satisfies the strong stationarity system \eqref{eq:strongstationarity} of \eqref{eq:P}
with a triple $(\bar p, \bar \nu, \bar \eta) \in H_0^1(\Omega) \times L^2(\Omega) \times H^{-1}(\Omega)$,
state $\bar y := S(\bar u) \in H_0^1(\Omega)$ and multiplier $\bar \lambda := -\Delta \bar y - \bar u \in L^2(\Omega)$. 
Then, for 
every $u \in U_{ad}$ with associated state $y := S(u)$ and multiplier $\lambda := -\Delta y - u$, it holds
\begin{equation}
\label{eq:Taylorexpansion}
\begin{aligned}
 J(y, u) - J(\bar y, \bar u)
&  =
\left \langle  \bar p  , \lambda - \bar \lambda \right \rangle +
\left \langle  \bar \eta, y - \bar y \right \rangle + ( \bar \nu, u - \bar u)_{L^2}
\\
&\quad\qquad+ \int_0^1 (1 - s) j''( (1 - s)\bar y+ s y )(y - \bar y)^2 \mathrm{d}s + \frac{\alpha}{2} \|u - \bar u\|_{L^2}^2.
\end{aligned}
\end{equation} 
Here, $\smash{j''(v)z^2}$ is short for $\smash{j''(v)(z,z)}$ for all $\smash{v, z \in H_0^1(\Omega)}$. 
\end{lemma}

\begin{proof}
From the fundamental theorem of calculus, we obtain
\begin{equation*}
\begin{aligned}
J(y, u) - J(\bar y, \bar u)
&= 
\left \langle j'(\bar y), y - \bar y \right \rangle + (\alpha \bar u, u - \bar u)_{L^2}
\\
&\qquad+ \int_0^1 (1 - s) j''( (1 - s)\bar y+ s y )(y - \bar y)^2 \mathrm{d}s + \frac{\alpha}{2} \|u - \bar u\|_{L^2}^2.
\end{aligned}
\end{equation*}
Further, \eqref{eq:strongstationarity} implies
\begin{equation*}
\begin{aligned}
\left \langle j'(\bar y), y - \bar y \right \rangle + (\alpha \bar u, u - \bar u)_{L^2}
&= 
\left \langle -\Delta \bar p + \bar \eta, y - \bar y \right \rangle + (-\bar p + \bar \nu, u - \bar u)_{L^2}
\\
&= 
\left \langle  \bar p  , \lambda - \bar \lambda \right \rangle +
\left \langle  \bar \eta, y - \bar y \right \rangle + ( \bar \nu, u - \bar u)_{L^2}
.
\end{aligned}
\end{equation*}
Combining the above two identities yields the claim.
\end{proof}
In the remainder of this paper, 
we frequently use the expansion \eqref{eq:Taylorexpansion} to derive estimates for the 
objective function of the problem \eqref{eq:P}. 
The next lemma collects some auxiliary
identities which turn out to be helpful in this context. 
\begin{lemma}
	\label{lem:aux_in_acht}
	In the situation of \cref{lemma:Taylorlikeexpansion},
	we have
	% Suppose that $\bar u \in U_{ad}$ satisfies the strong stationarity system \eqref{eq:strongstationarity} of \eqref{eq:P}
	% with a triple $(\bar p, \bar \nu, \bar \eta) \in H_0^1(\Omega) \times L^2(\Omega) \times H^{-1}(\Omega)$,
	% state $\bar y := S(\bar u) \in H_0^1(\Omega)$ and multiplier $\bar \lambda := -\Delta \bar y - \bar u \in L^2(\Omega)$. 
	% Then, for 
	% every $u \in U_{ad}$ with associated state $y := S(u)$ and multiplier $\lambda := -\Delta y - u$, it holds
	\begin{subequations}
		\label{eq:aux_in_acht}
		\begin{align}
			\label{eq:aux_in_acht_1}
			\left \langle \bar \eta, \min(0, y - \bar y)\right \rangle &= 0
			,
			\\
			\label{eq:aux_in_acht_3}
			\left \langle \bar \lambda, \min(0, y - \bar y)\right \rangle &= 0
			% \\
			% \label{eq:aux_in_acht_2}
			% \scalarprod{\bar p}{\bar\lambda}_{L^2} &= 0
		.
		\end{align}
		Further, it holds 
		\begin{equation}
			\label{eq:aux_in_acht_4}
			\begin{aligned}
				\left \langle  \bar p  , \lambda - \bar \lambda \right \rangle +
				\left \langle  \bar \eta, y  - \bar y \right \rangle
				&=
				\left ( \bar p  +  \beta ( \bar y - \psi)  , \lambda \right )_{L^2} 
				+ \left \langle  \bar \eta + \beta \bar \lambda  , \max\left (0, y  - \bar y \right ) \right \rangle 
				\\
				&\qquad+ \beta \left ( y - \bar y, \lambda - \bar \lambda\right )_{L^2}
			\end{aligned}
		\end{equation}
	\end{subequations}
	for all $\beta \in \R$.
\end{lemma}
\begin{proof}
	From Stampacchia's lemma, 
	see \cite[Proposition~5.8.2]{Attouch2006}, and 
	the inequalities $y \geq \psi$ and $\bar y \geq \psi$, we obtain that $\min(0, y - \bar y)$
	is an element of $H_0^1(\Omega)$ which vanishes quasi-everywhere on the active set $\{\bar y = \psi\}$
	(defined w.r.t.\ the continuous representatives). In tandem with 
	\cite[Theorem 9.1.3]{AdamsHedberg1999}, the continuity of the function $\bar y$, 
	and the properties of $\bar \eta$, this yields in particular
	that $\pm \min(0, y - \bar y) \in \TT_K(\bar y) \cap \bar \lambda^\perp$
	and 
	$\left \langle \bar \eta, \min(0, y - \bar y)\right \rangle = 0$.
	The identities \eqref{eq:aux_in_acht_1} and \eqref{eq:aux_in_acht_3} now follow immediately. 
	From \eqref{eq:strongstationarity_3} and the formula \eqref{eq:multiplierformula}, 
	we further obtain that $\dual{\bar p}{\bar\lambda} = 0$ and $(\psi,\lambda)_{L^2} = (y,\lambda)_{L^2}$.
	Together with \eqref{eq:aux_in_acht_1} and \eqref{eq:aux_in_acht_3},  the last two identities imply
	\begin{align*}
		&\left \langle  \bar p  , \lambda - \bar \lambda \right \rangle +
		\left \langle  \bar \eta, y  - \bar y \right \rangle
		\\
		&\qquad=
		\left ( \bar p  +  \beta ( \bar y - \psi)  , \lambda \right )_{L^2} 
		+ \left \langle  \bar \eta + \beta \bar \lambda  , y  - \bar y \right \rangle 
		-\beta \left (  \bar y - \psi, \lambda \right )_{L^2}
		-\beta \left \langle   \bar \lambda  , y  - \bar y \right \rangle
		\\
		&\qquad=
		\left ( \bar p  +  \beta ( \bar y - \psi)  , \lambda \right )_{L^2} 
		+ \left \langle  \bar \eta + \beta \bar \lambda  , \max\left (0, y  - \bar y \right ) \right \rangle 
		+ \beta \left ( y - \bar y, \lambda - \bar \lambda\right )_{L^2}
	\end{align*}
	for all $\beta \in \R$.
	This establishes \eqref{eq:aux_in_acht_4} and completes the proof. 
\end{proof}

Using \eqref{eq:Taylorexpansion}, we can prove the following theorem
that essentially combines the approaches of 
\cite{KunischWachsmuth2012,AliDeckelnickHinze2018,Mignot1976}
and extends the results of these papers to the control-constrained case:

\begin{theorem}[SSC Involving Compatibility Conditions]
\label{th:SSCmajor}
Suppose that $\bar u \in U_{ad}$ satisfies the strong stationarity system \eqref{eq:strongstationarity} of \eqref{eq:P}
with a triple $(\bar p, \bar \nu, \bar \eta) \in H_0^1(\Omega) \times L^2(\Omega) \times H^{-1}(\Omega)$,
state $\bar y := S(\bar u) \in H_0^1(\Omega)$ and multiplier $\bar \lambda := -\Delta \bar y - \bar u \in L^2(\Omega)$. 
Then, the following holds true:
\begin{enumerate}
\item \label{item:MajorSSC-i} 
If 
there exist constants $\beta \geq 0$ and $\gamma > 0$ with
\begin{equation}
\label{eq:multiplierassumptions-i}
\begin{aligned}
\bar p +  \beta ( \bar y - \psi) &\geq 0 \text{ a.e.\ in } \{0 < \bar y - \psi < \gamma\},
\\
\bar \eta + \beta \bar \lambda &\geq 0 \text{ in the sense of }H^{-1}(\Omega),
\end{aligned}
\end{equation}
and if 
\begin{equation}
\label{eq:ssc2534}
j''(\bar y)S'(\bar u; h)^2 + \alpha \|h\|_{L^2}^2 >  0
\end{equation}
holds for all $h \in \TT_{U_{ad}}(\bar u) \setminus \{0\}$ with
\begin{equation}
\label{eq:hproperties32}
h \in \bar \nu^\perp,\qquad 
-\Delta S'(\bar u; h) - h \in \bar p^\perp,\qquad S'(\bar u; h) \in \bar \eta^\perp,
\end{equation}
then $\bar u$ is locally optimal for \eqref{eq:P} and there exist constants $c, \varepsilon > 0$ with
\begin{equation}
\label{eq:SSCmajor1}
J(S(u), u) \geq J(S(\bar u), \bar u) + \frac{c}{2} \|u - \bar u\|_{L^2}^2\qquad \forall u \in U_{ad} \cap B_\varepsilon^{L^2}(\bar u).
\end{equation}

\item \label{item:MajorSSC-ii} 
If there exist constants $\beta \geq 0$ and $\mu \in \R$ such that 
\begin{equation}
\label{eq:stronglyconvex}
\begin{aligned}
\bar p +  \beta ( \bar y - \psi) &\geq 0 \text{ a.e.\ in } \Omega,
\\
\bar \eta + \beta \bar \lambda &\geq 0 \text{ in the sense of }H^{-1}(\Omega),
\\
j''(y)z^2 &\geq \mu \|z\|_{L^2}^2 \quad \forall (y, z) \in K \times (K - \bar y),
\end{aligned}
\end{equation}
and
\begin{equation}
\label{eq:constantestimate}
\mu + 2 \beta \omega - \frac{\beta^2}{\alpha} \geq 0
\end{equation}
holds,
where  $\omega > 0$ denotes the Poincaré constant of $\Omega$, i.e.,
\begin{equation}
\label{eq:PoincareDef}
\omega := \inf_{v \in H_0^1(\Omega) \setminus \{0\}} \frac{\int_\Omega |\nabla v|^2 \mathrm{d}x }{\int_\Omega v^2 \mathrm{d}x},
\end{equation}
 then $\bar u$ is globally optimal for \eqref{eq:P}.
If, moreover, the inequality \eqref{eq:constantestimate} is strict, then $\bar u$ is the unique global optimum of \eqref{eq:P},
and there exists a constant $c > 0$ with
\begin{equation}
\label{eq:globalgrowth22}
\begin{aligned}
&J(S(u), u) \geq J(S(\bar u), \bar u) 
+ \frac{c}{2}  \|u - \bar u\|_{L^2}^2  \qquad \forall u \in U_{ad}.
\end{aligned}
\end{equation}
\end{enumerate}
\end{theorem}

\begin{proof}
Ad \ref{item:MajorSSC-i}:
We follow the lines of \cite{KunischWachsmuth2012} and argue by contradiction
	(cf.\ also \cite{ChristofWachsmuth2018}).
	Suppose that \eqref{eq:multiplierassumptions-i} is satisfied, that 
	\eqref{eq:ssc2534} holds for all $h \in \TT_{U_{ad}}(\bar u) \setminus \{0\}$ with \eqref{eq:hproperties32},
	and that there are no $c>0$, $\varepsilon > 0$ with \eqref{eq:SSCmajor1}.
	Then, we can find sequences $\{u_n\} \subset U_{ad}$ and $\{c_n\} \subset \mathbb{R}^+$
	with
	\begin{equation}
	\label{eq:random2635}
		c_n \searrow 0
		,\quad
		\norm{u_n - \bar u }_{L^2} \to 0
		\quad\text{and}\quad
		J(S(u_n), u_n) - J(S(\bar u), \bar u) < \frac{c_n}{2} \|u_n - \bar u\|_{L^2}^2.
	\end{equation} 
	Define $y_n := S(u_n)$, $\lambda_n := -\Delta y_n - u_n$, $t_n := \norm{u_n - \bar u}_{L^2}$ and $h_n := (u_n - \bar u) / t_n$.
	Then, it holds $t_n \searrow 0$,  $h_n \in \R^+(U_{ad} - \bar u)$ and $\norm{h_n}_{L^2} = 1$, and we may assume w.l.o.g.\ 
	that the sequence $h_n$ converges weakly in $L^2(\Omega)$ to some $h \in \TT_{U_{ad}}(\bar u)$ for $n \to \infty$.
	Note that, due to the properties of $S$ in \cref{th:solutionmapproperties}
and
the compactness of the embedding $L^2(\Omega) \hookrightarrow H^{-1}(\Omega)$, 
	the convergence $h_n \weakly h$ in $L^2(\Omega)$ implies in particular that 
$(y_n - \bar y) / t_n$ converges strongly in $H_0^1(\Omega)$ to $S'(\bar u; h)$, cf.\ the 
results in \cite[Section 2.2.1]{BonnansShapiro2000}. 
	Using \eqref{eq:Taylorexpansion}, \eqref{eq:random2635} and the continuity of $j''$, we may now deduce that 
	\begin{equation}
	\label{eq:randomestimate434}
	\begin{aligned}
		0
		&\geq
		\frac{J( y_n, u_n) - J(\bar y, \bar u) - \frac{c_n}{2} \|t_n h_n \|_{L^2}^2}{t_n^2}
		\\
		& = \frac{1}{t_n^2} \Big (
	\left \langle  \bar p  , \lambda_n - \bar \lambda \right \rangle +
	\left \langle  \bar \eta, y_n  - \bar y \right \rangle + ( \bar \nu,  t_n h_n )_{L^2} \Big )
	+ \frac{1}{2} j''(\bar y) S'(\bar u; h)^2  + \frac{\alpha}{2} \| h_n\|_{L^2}^2 + \oo(1)
	,
	\end{aligned}
	\end{equation}
	where the Landau symbol refers to the limit $n \to \infty$. 
	Due to \eqref{eq:aux_in_acht_4} and $\norm{h_n}_{L^2} = 1$,
	the above yields
	\begin{equation}
		\label{eq:longlong2635}
		\begin{aligned}
			0 &\ge
			\frac{1}{t_n} \left (
				\left ( \bar p  +  \beta ( \bar y - \psi)  , \frac{\lambda_n}{t_n} \right )_{L^2} 
				+ \left \langle  \bar \eta + \beta \bar \lambda  , \max\left (0, \frac{y_n  - \bar y}{t_n} \right ) \right \rangle 
			+ ( \bar \nu,  h_n )_{L^2} \right )
			\\
			&\qquad 
			+ \beta \left ( \frac{y_n - \bar y}{t_n}, \frac{\lambda_n - \bar \lambda}{t_n}\right )_{L^2}
			+ \frac{1}{2} j''(\bar y) S'(\bar u; h)^2  + \frac{\alpha}{2} 
			+ \oo(1).
		\end{aligned}
	\end{equation}
	Note that the convergence $(y_n - \bar y)/ t_n \to S'(\bar u; h)$ in $H_0^1(\Omega)$ implies that the multipliers $\lambda_n$
	and $\bar \lambda$ satisfy 
	$(\lambda_n - \bar \lambda)/t_n \to -\Delta S'(\bar u; h) - h$ in $H^{-1}(\Omega)$ and that, as a consequence, we have
	\begin{equation*}
		\left ( \frac{y_n - \bar y}{t_n}, \frac{\lambda_n - \bar \lambda}{t_n}\right )_{L^2} \to 
		\left \langle -\Delta S'(\bar u; h) - h, S'(\bar u; h)\right \rangle = 0,
	\end{equation*} 
	where the last equality follows from the variational inequality \eqref{eq:VIdirdiff} 
	by choosing the test functions  $z = 0$ and $z = 2 S'(\bar u; h)$.
	If we use the above in \eqref{eq:longlong2635}, then we obtain
	\begin{equation}
	\label{eq:randomeq27363783}
	\begin{aligned}
		0
		&\geq
		\limsup_{n \to \infty}  \Bigg ( \frac{1}{t_n} \Bigg [
	\left ( \bar p  +  \beta ( \bar y - \psi)  , \frac{\lambda_n}{t_n} \right )_{L^2} 
		+ \left \langle  \bar \eta + \beta \bar \lambda  , \max\left (0, \frac{y_n  - \bar y}{t_n} \right ) \right \rangle 
	+ ( \bar \nu,  h_n )_{L^2} \Bigg ] \Bigg )
	\\
	&\qquad\qquad\qquad\qquad\qquad
	+ \frac{1}{2} j''(\bar y)  S'(\bar u; h)^2  + \frac{\alpha}{2}.
	\end{aligned}
	\end{equation}
	Since the global Lipschitz continuity of the map $S : L^2(\Omega) \to L^\infty(\Omega)$
	and the properties of $h_n$ imply
	\begin{equation*}
		|y_n - \psi| \geq | \bar y - \psi| - \|y_n - \bar y\|_{L^\infty}
		\geq | \bar y - \psi| - C t_n
	\end{equation*}
	with an absolute constant $C>0$, we may further use \eqref{eq:multiplierformula}
	to deduce that $\lambda_n$ vanishes a.e.\ in the set $\{\bar y - \psi >C t_n\}$.
	If we combine this observation with our assumptions in \eqref{eq:multiplierassumptions-i}, then \eqref{eq:randomeq27363783} yields
	 (due  to the non-negativity of
	the terms in the square brackets for large $n$, the factor $1/t_n$, the continuity 
	of the map $H_0^1(\Omega) \ni z \mapsto \max(0, z) \in H_0^1(\Omega)$,
	the weak lower semicontinuity of 
	continuous and convex functions,
	and the properties of $S$, $\lambda_n$, etc.) that
	\begin{equation*}
	\begin{aligned}
	0 &= \lim_{n \to \infty} \left ( \bar p  +  \beta ( \bar y - \psi)  , \frac{\lambda_n}{t_n}\right )_{L^2} 
	\\
	 &= \lim_{n \to \infty} \left ( \bar p  +  \beta \max(0,  \bar y - y_n)  , \frac{\lambda_n - \bar \lambda}{t_n}\right )_{L^2}  
	= \left \langle - \Delta S'(\bar u; h) - h,  \bar p \right \rangle,
	\\
	0 &= \lim_{n \to \infty} \left \langle  \bar \eta + \beta \bar \lambda  , 
	\max \left (0, \frac{y_n  - \bar y}{t_n} \right ) \right \rangle
	= \lim_{n \to \infty} \left \langle  \bar \eta + \beta \bar \lambda  , 
	 \frac{y_n  - \bar y}{t_n} \right \rangle
	  = \left \langle  \bar \eta , 
	 S'(\bar u; h) \right \rangle,
	\\
	0 &=  \lim_{n \to \infty} ( \bar \nu,   h_n )_{L^2} = ( \bar \nu,   h)_{L^2},
	\end{aligned}
	\end{equation*}
	and
	\begin{equation*}
		0
		\geq
 		\frac{1}{2} j''(\bar y)  S'(\bar u; h)^2  + \frac{\alpha}{2}
		\geq
 		\frac{1}{2} j''(\bar y)  S'(\bar u; h)^2  + \frac{\alpha}{2} \|h\|_{L^2}^2.
	\end{equation*}
	Due to \eqref{eq:ssc2534} for all $h \in \TT_{U_{ad}}(\bar u) \setminus \{0\}$ with \eqref{eq:hproperties32},
	the above is impossible. Thus, we indeed arrive at a contradiction and the proof of the first 
	assertion is complete. 

Ad \ref{item:MajorSSC-ii}:
The proof is completely analogous to \cite{AliDeckelnickHinze2018}: 
From \eqref{eq:Taylorexpansion} and \eqref{eq:stronglyconvex}, it follows straightforwardly that,
for all $u \in U_{ad}$ with associated state $y = S(u)$ and multiplier $\lambda$, we have
\begin{equation*}
J(y, u) - J(\bar y, \bar u) 
\ge \left \langle  \bar p  , \lambda - \bar \lambda \right \rangle +
\left \langle  \bar \eta, y - \bar y \right \rangle + ( \bar \nu, u - \bar u)_{L^2}
+ \frac{\mu }{2}\|y - \bar y\|_{L^2}^2 + \frac{\alpha}{2} \|u - \bar u\|_{L^2}^2.
\end{equation*}
Using \eqref{eq:aux_in_acht_4}, \eqref{eq:strongstationarity_5}, 
the sign conditions in \eqref{eq:stronglyconvex},
% exactly the same arguments as in \ref{item:MajorSSC-i},
Young's inequality and the definitions of $\lambda$ and $\bar \lambda$,
we may now deduce that, for every arbitrary but fixed $\varepsilon \in [0,1]$, we have 
\begin{equation}
\label{eq:globaloptimalityestimate}
\begin{aligned}
&J(y, u) - J(\bar y, \bar u) 
% \\
% &\quad
% \geq \left \langle  \bar p  , \lambda \right \rangle +
% \left \langle  \bar \eta, \max(0, y - \bar y) \right \rangle + ( \bar \nu, u - \bar u)_{L^2}
% + \frac{\mu }{2}\|y - \bar y\|_{L^2}^2 + \frac{\alpha}{2} \|u - \bar u\|_{L^2}^2
\\
&\quad
\geq \beta \left \langle \lambda - \bar \lambda,  y - \bar y \right \rangle 
+ \frac{\mu }{2}\|y - \bar y\|_{L^2}^2 + \frac{\alpha}{2} \|u - \bar u\|_{L^2}^2
\\
&\quad
= \beta \int_\Omega \left  | \nabla y - \nabla \bar y \right |^2 \dx
+
\beta \left \langle \bar u - u,  y - \bar y \right \rangle 
+ \frac{\mu }{2}\|y - \bar y\|_{L^2}^2 + \frac{\alpha}{2} \|u - \bar u\|_{L^2}^2
\\
&\quad
\geq \int_\Omega  \left ( \beta \omega + \frac{\mu }{2}\right) (y - \bar y)^2 +
\beta (y - \bar y)(\bar u - u)
 + \frac{\alpha}{2} (u - \bar u)^2
\mathrm{d}x
\\
&\quad
\geq  \left [  \frac{1}{2} \left ( 2\beta \omega +  \mu  - \frac{\beta^2}{ \alpha} \right ) 
- \varepsilon\frac{\beta^2}{2 \alpha} \right] \| y - \bar y\|_{L^2}^2
 + \frac{\varepsilon \alpha}{2(1 + \varepsilon)} \| u - \bar u\|_{L^2}^2.
\end{aligned}
\end{equation}
Suppose now that the condition in \eqref{eq:constantestimate} is satisfied. Then, by choosing $\varepsilon = 0$
in \eqref{eq:globaloptimalityestimate}, we obtain immediately that $\bar u$ is a global optimum of  \eqref{eq:P}.
This proves the first assertion in \ref{item:MajorSSC-ii}. If, additionally, \eqref{eq:constantestimate}
holds with strict inequality, then we can choose a sufficiently small $\varepsilon > 0$ in \eqref{eq:globaloptimalityestimate} to arrive at \eqref{eq:globalgrowth22}.
The second assertion in \ref{item:MajorSSC-ii} now follows immediately. This completes the proof of the theorem. 
\end{proof}

Some remarks are in order regarding the last result:

\begin{remark}~
\begin{itemize}
\item Note that, for all $\bar u \in U_{ad}$, which are strongly stationary for \eqref{eq:P}, and all
 ${h \in \TT_{U_{ad}}(\bar u)}$, we have (cf.\ \eqref{eqrandomeq82636})
\begin{equation*}
\begin{aligned}
\left \langle j'(\bar y), S'(\bar u; h) \right \rangle + \alpha \left ( \bar u, h \right )_{L^2} 
&= \left \langle -\Delta \bar p + \bar \eta, S'(\bar u; h) \right \rangle + \left ( -\bar p + \bar \nu, h \right )_{L^2} 
\\
&=  \left \langle -\Delta  S'(\bar u; h) - h, \bar p\right \rangle 
+ \left \langle \bar \eta, S'(\bar u; h) \right \rangle
+ \left (\bar \nu, h \right )_{L^2}.
\end{aligned}
\end{equation*}
The above implies, in combination with the conditions in \eqref{eq:strongstationarity},
that a direction $h \in \TT_{U_{ad}}(\bar u)$ satisfies \eqref{eq:hproperties32}
if and only if $\left \langle j'(\bar y), S'(\bar u; h) \right \rangle + \alpha \left ( \bar u, h \right )_{L^2} = 0$.
This shows that, as usual in the analysis of second-order optimality conditions, 
\eqref{eq:ssc2534} is a positivity condition on the critical cone (without zero), i.e., on the set of all directions 
which satisfy the Bouligand stationarity condition \eqref{eq:Bouligandstatcond} with equality. 

\item It is easy to check that, in the situation of
\cref{th:SSCmajor}\ref{item:MajorSSC-ii},
\eqref{eq:strongstationarity_3}, \eqref{eq:strongstationarity_4}
and
the first two lines in \eqref{eq:stronglyconvex} 
can be recast as
\begin{equation*}
\bar p \in \R^+ \left ( K - \bar y\right ) \cap \bar\lambda\anni,\qquad \bar \eta \in \R^+\left (-\TT_K(\bar y)^\circ - \bar \lambda \right ),
\end{equation*}
where $\TT_K(\bar y)^\circ$ denotes the polar cone of $\TT_K(\bar y)$.
From \eqref{eq:strongstationarity} alone, we only obtain that $\bar p  \in \TT_K(\bar y) \cap\bar\lambda\anni \supset \R^+ \left ( K - \bar y\right ) \cap\bar\lambda\anni $
and 
$\bar \eta \in -\left ( \TT_K(\bar y) \cap \bar \lambda^\perp\right )^\circ \supset\R^+\left (-\TT_K(\bar y)^\circ - \bar \lambda \right )$. 
The assumptions on $\bar p$ and $\bar \eta$ in \cref{th:SSCmajor}\ref{item:MajorSSC-ii} thus
express that $\bar p$ and $\bar \eta$ satisfy stricter inclusions than those implied 
by the strong stationarity system \eqref{eq:strongstationarity}.

\item Observe that, in \cref{th:SSCmajor}\ref{item:MajorSSC-ii}, the functional $j$ is allowed to possess 
negative curvature if $\beta$, $\omega$ and $\alpha$ are suitable. 

\item Note that the conditions in \eqref{eq:multiplierassumptions-i} are indeed weaker than 
the non-negativity assumptions used in \cite[Section 2.2]{KunischWachsmuth2012} 
(due to the signs of $\bar y - \psi$ and $\bar \lambda$)
and that 
\cref{th:SSCmajor} indeed generalizes \cite[Theorem 3.2]{AliDeckelnickHinze2018}
where only global optima and problems without control constraints are considered. 
\end{itemize}
\end{remark}

In the remainder of this paper, our aim will be to derive
SSC for the problem \eqref{eq:P} that 
involve more tangible/milder assumptions on the relationship between 
$\bar p$, $\bar \eta$, $\bar \lambda$ and $\bar y$ than those in \eqref{eq:multiplierassumptions-i} and 
\eqref{eq:stronglyconvex}.
To achieve this goal, we will study in more detail the structure 
of the stationary points $\bar u$ of \eqref{eq:P}
and the form of the associated multipliers $\bar \lambda$.

\section{Structure of Optimal Controls and Identification with a State-Constrained Optimal Control Problem}
\label{sec:4}

The main idea of the analysis in the next three sections 
is to exploit that the Ti\-kh\-on\-ov regularization term $\frac{\alpha}{2}\|u\|_{L^2}^2$
imposes a special structure on the minimizers and Bouligand stationary points $\bar u \in U_{ad}$
of the problem \eqref{eq:P}. As we will see, this special structure 
makes it possible to recast \eqref{eq:P} as a
state-constrained optimal control problem for the Poisson equation 
(with a modified objective function) and to derive sufficient conditions 
for local and global optimality in a very natural way.
Since the subsequent analysis is completely unaffected 
by the presence of additional state constraints,
in this section and the next, we also allow that the optimal 
control problem under consideration contains a condition of the form $y \in Y_{ad}$.
To be more precise, we assume that a problem of the type
\begin{equation}
	\label{eq:Q}
	\tag{\textup{Q}}
	\begin{aligned}
		\text{Minimize} \quad & J(y, u) := j(y) + \frac{\alpha}{2}\|u\|_{L^2}^2 \\
		\text{w.r.t.}\quad &(y,u) \in H_0^1(\Omega) \times L^2(\Omega) \\
		\text{s.t.}\quad & y \in K, \ \ \left \langle -\Delta y, v - y \right \rangle \geq \left \langle u, v - y \right \rangle \ \ \forall v \in K \\
		& u \in U_{ad} := \left \{ w \in L^2(\Omega) \mid u_a \leq w \leq u_b \text{ a.e.\ in } \Omega \right \} \\
		\text{and}\quad& y \in Y_{ad}
	\end{aligned}
\end{equation}
is given and that the following is satisfied:

\begin{assumption}[Standing Assumptions for the Study of Problem {\eqref{eq:Q}}]~%
\label{assumption:standingQ}%
\begin{itemize}
\item $d$, $\Omega$, $j$, $\alpha$, $K$, $\psi$, $u_a$ and $u_b$ are as in \cref{assumption:standing},
\item $Y_{ad}$ is a weakly closed subset of $H^2(\Omega)$ and there exists a control $u \in U_{ad}$ with $S(u) \in Y_{ad}$. 
\end{itemize}
\end{assumption}

Let us briefly check that the additional constraint $y \in Y_{ad}$ in \eqref{eq:Q} 
has no effect on the well-posedness of the problem:

\begin{proposition}[Solvability and Bouligand Stationarity for \eqref{eq:Q}]
\label{prop:Qwellposedness}
The optimal control problem \eqref{eq:Q}  admits at least one global 
solution $\bar u \in L^2(\Omega)$. Moreover, every local solution $\bar u \in L^2(\Omega)$
of \eqref{eq:Q} with state $\bar y := S(\bar u)$ satisfies the Bouligand stationarity condition 
\begin{equation}
\label{eq:BouligandstatcondQ}
\left \langle j'(\bar y), S'(\bar u; h) \right \rangle + \alpha \left ( \bar u, h \right )_{L^2} \geq 0 \qquad 
\forall h \in 
\TT_{\UU_{ad}}^{w{\text -}out}(\bar u).
\end{equation}
Here, $\TT_{\UU_{ad}}^{w{\text -}out}(\bar u)$ denotes the weak outer tangent cone 
of the (not necessarily convex) effective admissible set 
$\UU_{ad} := \{u \in U_{ad} \mid S(u) \in Y_{ad}\}$ of \eqref{eq:Q} at $\bar u$, i.e., 
\begin{equation*}
\begin{aligned}
 \TT_{\UU_{ad}}^{w{\text -}out}(\bar u)
 :=
\set*{ 
h \in L^2(\Omega)
\given
\exists t_n \searrow 0, h_n \overset{L^2}{\weakly} h
\text{ such that }  \bar u + t_n h_n \in \UU_{ad} \ \forall n \in \mathbb{N}
}.
\end{aligned}
\end{equation*}
\end{proposition}

\begin{proof}
The existence of a global solution $\bar u$ and the stationarity condition \eqref{eq:BouligandstatcondQ}
follow from the 
direct method of calculus of variations, the properties of the quantities in \eqref{eq:Q},
the estimate \eqref{eq:H2reg}, the mapping properties of the operator $S$,
the definition of $\TT_{\UU_{ad}}^{w{\text -}out}(\bar u)$ and a simple calculation. 
\end{proof}

Note that, in the special case $Y_{ad} = H^2(\Omega)$, the Bouligand stationarity condition \eqref{eq:BouligandstatcondQ}
takes precisely the form \eqref{eq:Bouligandstatcond} (due to the lemma of Mazur).
\Cref{prop:Qwellposedness} is thus consistent with the results that we have 
established in \cref{sec:2} for the problem \eqref{eq:P}.
The key observation is now the following: 

\begin{theorem}[Structure of Stationary Points and Partially Optimal Controls]
~\label{the:Qstructure}
\begin{enumerate}
\item \label{th:Qstructure:i}
Suppose that
$\bar u \in U_{ad}$ 
is a Bouligand stationary point of the problem \eqref{eq:Q}
with state $\bar y := S(\bar u) \in Y_{ad}$ 
and multiplier $\bar \lambda := -\Delta \bar y - \bar u \in L^2(\Omega)$.
Then, it necessarily holds
\begin{equation*}
\bar u =
\begin{cases}
\min(0, -\Delta \psi) & \text{a.e.\ in } \{\bar y = \psi\}
\\
-\Delta \bar y &\text{a.e.\ in } \{\bar y > \psi\}  
\end{cases}
,
\quad
\bar \lambda
=
\begin{cases}
\max(0, -\Delta \psi)& \text{a.e.\ in } \{\bar y = \psi\}
\\
0 &\text{a.e.\ in } \{\bar y > \psi\}  
\end{cases}.
\end{equation*}
In particular, in addition to the complementarity condition  $0 \leq \bar y - \psi \perp \bar \lambda \geq 0$ a.e.\ in~$\Omega$
associated with the obstacle problem, the triple $(\bar u,\bar y,\bar \lambda)$ satisfies 
\begin{equation}
\label{eq:doublecomplementarity}
0 \leq - \bar u \perp \bar \lambda \geq 0 \quad \text{a.e.\ in } \{\bar y = \psi\}.
\end{equation}
Here, $a \perp b$, $a, b \in \R$, means that at least one of the numbers $a$ and $b$ is zero. 

\item \label{th:Qstructure:ii}
Suppose that $y \in Y_{ad}$ is a state that is attainable in \eqref{eq:Q} (i.e., a state such that there exists a control
$u \in U_{ad}$ with $S(u) = y$). Then, the function 
\begin{equation}
\label{eq:uydef}
\begin{aligned}
u_y & :=
\begin{cases}
\min(0, -\Delta \psi) & \text{a.e.\ in } \{y = \psi\}
\\
- \Delta y &\text{a.e.\ in } \{y > \psi\}
\end{cases}
\end{aligned}
\end{equation}
satisfies $u_y \in U_{ad}$, $S(u_y) = y \in Y_{ad}$ and 
\begin{equation}
\label{eq:randomeq287364782}
	\norm{u}_{L^2}^2
	-
	\norm{u_y}_{L^2}^2
	\ge
	\norm{u - u_y}_{L^2}^2\quad \forall u \in U_{ad} \text{ with } S(u) = y.
\end{equation}
In particular,
\begin{equation}
\label{eq:partialoptimality}
\{u_y\} = \argmin_{u \in U_{ad},\, S(u) = y} J(y, u).
\end{equation}
\end{enumerate}
\end{theorem}

\begin{proof}
Ad \ref{th:Qstructure:i}: Suppose that an arbitrary but fixed Bouligand stationary 
point $\bar u \in  U_{ad}$ with state $\bar y \in Y_{ad}$ and multiplier $\bar \lambda \in L^2(\Omega)$ is given.
Then, the function
\begin{equation*}
\begin{aligned}
h := 
&\mathds{1}_{\{\bar y = \psi\}}\Big (\min(0, -\Delta \psi) - \bar u \Big ) \in L^2(\Omega)
\end{aligned}
\end{equation*}
satisfies
\begin{equation*}
\bar u + t h =
\begin{cases}
\bar u & \text{ a.e.\ in } \{\bar y > \psi\}
\\
(1 - t)\bar u + t \min(0, -\Delta \psi) & \text{ a.e.\ in } \{\bar y = \psi\}
\end{cases}
\end{equation*}
and (due to  \eqref{eq:multiplierformula})
\begin{equation*}
-\Delta \bar y - (\bar u + t h )
=
\begin{cases}
0 & \text{a.e.\ in } \{\bar y > \psi \}  
\\
t \max (0, - \Delta \psi ) + (1 - t)(-\Delta \psi- \bar u) & \text{a.e.\ in } \{\bar y = \psi \}  
\end{cases}
\end{equation*}
for all $t \in [0,1]$.
Since $u_a \leq 0 \leq u_b$ a.e.\ in $\Omega$, 
$0 \leq \bar \lambda = - \Delta \psi - \bar u$ a.e.\ in $\{\bar y = \psi\}$,
and
$\bar u \leq \bar u + t \bar \lambda = (1-t)\bar u - t \Delta \psi \leq - \Delta \psi \leq 0$ a.e.\ in $\{\bar y = \psi,\Delta \psi \geq 0\}$
for all $t \in [0,1]$,
the above identities imply $u_a \leq \bar u + t h \leq u_b$ and $-\Delta \bar y - (\bar u + t h ) \geq 0$ a.e.\ in $\Omega$
for all $t \in [0,1]$.
In particular, 
it holds $\bar u + th \in U_{ad}$ and $S(\bar u + th ) = S(\bar u) = \bar y \in Y_{ad}$ for all $t \in [0,1]$
by the definition of the set $U_{ad}$ and the variational inequality \eqref{eq:obstacleproblem}.
From the definitions of $\UU_{ad}$, $\TT_{\UU_{ad}}^{w{\text -}out}(\bar u)$ and $h$,
the fact that $S(\bar u + th ) - S(\bar u) = 0$ for all $t \in [0,1]$,
the Bouligand stationarity condition \eqref{eq:BouligandstatcondQ},
and again the properties of $\bar \lambda$ and $\bar u$,
we may now deduce that
\begin{equation*}
\begin{aligned}
0 &\leq \left \langle j'(\bar y), S'(\bar u; h) \right \rangle + \alpha \left ( \bar u, h \right )_{L^2}
% \\&
= 0 +  \alpha \left ( \bar u, h \right )_{L^2}
\\
& = - \alpha  \int_{\{\bar y = \psi , \Delta \psi \leq 0\}} \bar u^2 \mathrm{d}x
+
\alpha  \int_{\{\bar y = \psi , \Delta \psi > 0\}} \bar u \left ( -\Delta \psi - \bar u \right )  \mathrm{d}x
\\
&=
 -\alpha \int_{\{\bar y = \psi , \Delta \psi \leq 0\}} \bar u^2 \mathrm{d}x
-
\alpha  \int_{\{\bar y = \psi , \Delta \psi > 0\}} |\bar \lambda + \Delta \psi| \left | -\Delta \psi - \bar u \right |  \mathrm{d}x \leq 0.
\end{aligned}
\end{equation*}
This establishes the formula for $\bar u$. 
The formula for $\bar \lambda$
and the complementarity relation $0 \leq - \bar u \perp \bar \lambda \geq 0$ a.e.\ in $\{\bar y = \psi\}$
now follow immediately from the identity $\bar \lambda = - \Delta \bar y- \bar u$
and a simple computation. 
This completes the proof of the first part of the theorem. 

Ad \ref{th:Qstructure:ii}:
If we are given an arbitrary but fixed control $u \in U_{ad}$ with associated state $y = S(u) \in Y_{ad}$,
then we can use exactly the same calculation as in  \ref{th:Qstructure:i} (with $t=1$)
to prove that the function $u_y = u + \mathds{1}_{\{y = \psi\}}  (\min(0, -\Delta \psi) - u  )$
satisfies $u_y \in U_{ad}$ and $S(u_y) = y \in Y_{ad}$. 
Since 
$u_y = u = -\Delta y$ a.e.\ in $\{y > \psi \}$,
$u_y = 0$ a.e.\ in $\{y = \psi, \Delta \psi \leq 0\}$,
$u_y = - \Delta \psi \leq 0$ a.e.\ in $\{y = \psi, \Delta \psi > 0\}$,
and
$u \leq - \Delta \psi \leq 0$ a.e.\ in $\{y = \psi,\Delta \psi \geq 0\}$
(see part \ref{th:Qstructure:i}), we may further calculate that
\begin{equation*}
\begin{aligned}
\| u - u_y\|_{L^2}^2
&=
\int_{\{y = \psi, \Delta \psi \leq 0\}} u^2 \mathrm{d}x
+
\int_{\{y = \psi, \Delta \psi > 0\}} (u + \Delta \psi)^2 \mathrm{d}x
\\
&=
\int_{\{y = \psi, \Delta \psi \leq 0\}} u^2 \mathrm{d}x
+
\int_{\{y = \psi, \Delta \psi > 0\}}  u^2 + 2 \Delta \psi (u   + \Delta \psi) - (\Delta \psi)^2\mathrm{d}x
\\
&\leq 
\int_{\{y = \psi\}} u^2 \mathrm{d}x - \int_{\{y = \psi, \Delta \psi > 0\}} (\Delta \psi )^2 \mathrm{d}x
= \| u \|_{L^2}^2 - \| u_y\|_{L^2}^2.
\end{aligned}
\end{equation*}
The assertions in \ref{th:Qstructure:ii} now follow immediately. 
\end{proof}

The main point of \cref{the:Qstructure}
is that, for every arbitrary but fixed attainable state $y \in Y_{ad} \cap S(U_{ad})$, 
there is one and only one admissible control 
that is relevant for the analysis of \eqref{eq:Q}, namely the function $u_y$ in \eqref{eq:uydef}. 
All other controls $u \in U_{ad}$ with $y = S(u)$ are suboptimal for \eqref{eq:Q} by \eqref{eq:partialoptimality} and can 
be neglected. 

We would like to point out that the effect that we observe here is a direct consequence 
of the Ti\-kh\-on\-ov regularization term present in the objective function $j(y) +  \frac{\alpha}{2}\|u\|_{L^2}^2$
 of \eqref{eq:Q}.
To see this, recall that, without the Ti\-kh\-on\-ov regularization, i.e., in the case $\alpha = 0$,
the problem \eqref{eq:Q} is maximally ill-posed since for a given state $y \in Y_{ad} \cap S(U_{ad})$
 there are typically infinitely many controls $u$
with $y = S(u)$, namely, all those $u \in U_{ad}$ with
\begin{equation*}
u  
\leq - \Delta \psi   \text{ a.e.\ in } \{y = \psi\}
\qquad\text{and}\qquad 
u
= - \Delta y   \text{ a.e.\ in } \{y > \psi\}.
\end{equation*}
The Ti\-kh\-on\-ov regularization $\frac{\alpha}{2}\|u\|_{L^2}^2$ resolves the above ambiguity by 
making one control energetically more favorable than the others.
This distinguished control is precisely the function $u_y$ that we have calculated in \eqref{eq:uydef}. 
Note that, since the ``partially optimal'' control $u_y$ is uniquely determined 
by the state $y$, the multiplier $\lambda_y$ that is associated with the tuple 
$(y,u_y)$ can be expressed in terms of $y$ as well. Indeed, by using the identity $\lambda_y = -\Delta y - u_y$,
we obtain (analogously to part \ref{th:Qstructure:i} of \cref{the:Qstructure})
\begin{equation}
\label{eq:lambdaydef}
\lambda_y
=
\begin{cases}
\max(0, -\Delta \psi)& \text{a.e.\ in } \{y = \psi\}
\\
0 &\text{a.e.\ in } \{ y > \psi\}  
\end{cases}.
\end{equation}
From the properties of the controls $u_y$ in \cref{the:Qstructure}, we may now deduce:

\begin{corollary}[Reduction to Partially Optimal Controls]
\label{cor:redux1}
Let $\UU_{ad}$ denote the 
effective admissible set of \eqref{eq:Q}, i.e., $\UU_{ad} := \{u \in U_{ad} \mid S(u) \in Y_{ad}\}$.
Suppose further
that a control $\bar u \in \UU_{ad}$ with state $\bar y := S(\bar u)$ is 
given such that $\bar u = u_{\bar y}$ holds, where $u_y$ is defined by \eqref{eq:uydef}
for all $y \in S(\UU_{ad})$. Then, the following holds true:
\begin{enumerate}

\item 
\label{cor:redux1:i}
The existence of a constant $r_1 > 0$ with
\begin{equation}
\label{eq:locoptimality1}
J(S(u), u) \geq J(S(\bar u), \bar u) \qquad \forall u \in \UU_{ad} \cap B_{r_1}^{L^2}(\bar u)
\end{equation}
is equivalent to the existence of a constant $r_2 > 0$ with
\begin{equation}
\label{eq:locoptimality2}
J(S(u), u) \geq J(S(\bar u), \bar u)  \qquad \forall u \in \left \{ u_y \mid y \in S(\UU_{ad}) \right \} \cap B_{r_2}^{L^2}(\bar u).
\end{equation}

\item 
\label{cor:redux1:ii}
The existence of constants $c_1 > 0$, $r_1 > 0$ with
\begin{equation}
\label{eq:quadgrowth1}
J(S(u), u) \geq J(S(\bar u), \bar u)  + \frac{c_1}{2}\|u - \bar u\|_{L^2}^2\qquad \forall u \in \UU_{ad} \cap B_{r_1}^{L^2}(\bar u)
\end{equation}
is equivalent to the existence of constants $c_2 >0$, $r_2 > 0$ with
\begin{equation}
\label{eq:quadgrowth2}
J(S(u), u) \geq J(S(\bar u), \bar u)  + \frac{c_2}{2}\|u - \bar u\|_{L^2}^2\qquad  
\forall u \in \left \{ u_y \mid y \in S(\UU_{ad}) \right \} \cap B_{r_2}^{L^2}(\bar u).
\end{equation}

\item 
\label{cor:redux1:iii}
The estimate \eqref{eq:locoptimality1} (respectively, \eqref{eq:quadgrowth1}) holds with $r_1 = \infty$
if and only if 
the estimate \eqref{eq:locoptimality2} (respectively, \eqref{eq:quadgrowth2}) holds with $r_2 = \infty$.
\end{enumerate}
\end{corollary}

\begin{proof}
To prove \cref{cor:redux1}, we first note that, if
a sequence $\{u_n\} \subset \UU_{ad}$ 
with associated states $y_n := S(u_n) \in Y_{ad}$
converges to $\bar u$, then the sequence of partially optimal controls $\{u_{y_n}\}\subset \UU_{ad}$
converges to $\bar u$ as well. Indeed, for every $\{u_n\} \subset \UU_{ad}$ with $u_n \to \bar u$ in $L^2(\Omega)$ and $y_n:=S(u_n)$,
\eqref{eq:uydef}, the continuity of the map $S : u \mapsto y$ as a function
from $L^2(\Omega)$ to $L^\infty(\Omega)$, the dominated convergence theorem, and our assumption $\bar u = u_{\bar y}$ yield
\begin{equation}
\label{eq:randomeq29738247}
\begin{aligned}
 \left \|
u_{y_n}
-
\bar u
\right \|_{L^2}
&\leq 
\left \|
\mathds{1}_{\{y_n > \psi\}} ( u_n- \bar u)
\right \|_{L^2}
+
\left \|
\mathds{1}_{\{y_n = \psi\}} (\min(0, - \Delta \psi) - \bar u)
\right \|_{L^2}
\\
& 
\leq 
\left \| u_n - \bar u 
\right \|_{L^2}
+
\left \|
\mathds{1}_{\{y_n = \psi, \,  \bar y > \psi \}}(\min(0, - \Delta \psi) - \bar u)
\right \|_{L^2} \to 0
\end{aligned}
\end{equation}
for $n \to \infty$. The claims in \ref{cor:redux1:i}, \ref{cor:redux1:ii}, and \ref{cor:redux1:iii} can now be established as follows:

Ad \ref{cor:redux1:i}: The implication \eqref{eq:locoptimality1} $\Rightarrow$ \eqref{eq:locoptimality2} is trivial.
To establish \eqref{eq:locoptimality2} $\Rightarrow$  \eqref{eq:locoptimality1}, we argue by contradiction:
Suppose that \eqref{eq:locoptimality2} holds with some $r_2 > 0$ and that \eqref{eq:locoptimality1} is violated. Then, we
can find a sequence $\{u_n\} \subset \UU_{ad}$ with $u_n \to \bar u$ in $L^2(\Omega)$ and states $y_n := S(u_n)$
such that $J(y_n, u_n) < J(\bar y, \bar u)$ holds for all $n$. From \eqref{eq:randomeq287364782} and the definition of $J$, we now obtain
\begin{equation*}
J(\bar y, \bar u) > J(y_n, u_n) \geq J(y_n, u_{y_n})
\end{equation*}
for all $n$. Since $u_{y_n} \to \bar u$ in $L^2(\Omega)$, the above contradicts  \eqref{eq:locoptimality2}.
 This proves \ref{cor:redux1:i}.

Ad \ref{cor:redux1:ii}: The proof of \ref{cor:redux1:ii} is along the lines of that of \ref{cor:redux1:i}:
The implication \eqref{eq:quadgrowth1} $\Rightarrow$ \eqref{eq:quadgrowth2} is trivial. 
If \eqref{eq:quadgrowth2} holds with some $c_2 > 0$, $r_2 > 0$, but \eqref{eq:quadgrowth1} is violated,
then we can find a sequence  $\{u_n\} \subset \UU_{ad}$ with associated states $y_n := S(u_n)$ such that $u_n$
converges to $\bar u$ in $L^2(\Omega)$
and such that 
\begin{equation*}
J(y_n, u_n) < J(\bar y, \bar u)  + \frac{1}{2n}\|u_n - \bar u\|_{L^2}^2 
\end{equation*}
holds for all $n$. From \eqref{eq:randomeq287364782}, the definition of $J$, the elementary estimate $a^2 + b^2 \geq (a+b)^2/2$
for all $a, b \in \R$, \eqref{eq:quadgrowth2}, and the convergence $u_{y_n} \to \bar u$ in $L^2(\Omega)$,
it now follows  that
\begin{equation*}
\begin{aligned}
J(\bar y, \bar u)  &> J(y_n, u_n) - \frac{1}{2n}\|u_n - \bar u\|_{L^2}^2
\\
&\geq J(y_n, u_{y_n}) +  \frac{\alpha}{2}\|u_n - u_{y_n}\|_{L^2}^2 - \frac{1}{2n}\|u_n - \bar u\|_{L^2}^2
\\
&\geq J(\bar y, \bar u) +  \frac{c_2}{2}\| u_{y_n} - \bar u\|_{L^2}^2 + \frac{\alpha}{2}\|u_n - u_{y_n}\|_{L^2}^2 - \frac{1}{2n}\|u_n - \bar u\|_{L^2}^2
\\
&\geq J(\bar y, \bar u) +  \frac{\min(\alpha, c_2)}{4}\| u_{n} - \bar u\|_{L^2}^2 - \frac{1}{2n}\|u_n - \bar u\|_{L^2}^2
\end{aligned}
\end{equation*}
holds for all large enough $n$. 
This again yields a contradiction.

Ad \ref{cor:redux1:iii}: The assertions in \ref{cor:redux1:iii} follow straightforwardly from
\eqref{eq:randomeq287364782} and \eqref{eq:partialoptimality}.
\end{proof}

Note that \cref{the:Qstructure} and \cref{cor:redux1} yield that, as far as local/global optima and local/global quadratic growth conditions 
are concerned, instead of the original optimal control problem \eqref{eq:Q}, we can also study the reduced 
minimization problem
\begin{equation}
	\label{eq:ReducedProblem}
	\begin{aligned}
		\text{Minimize} \quad & J(y, u_y) := j(y) + \frac{\alpha}{2}\|u_y\|_{L^2}^2 \\
		\text{w.r.t.}\quad &(y,u_y) \in H_0^1(\Omega) \times L^2(\Omega) \\
		\text{s.t.} \quad& u_y \in U_{ad}, \ \ y \in Y_{ad} \cap K
	\end{aligned}
\end{equation}
with $u_y$ defined as in \eqref{eq:uydef}. 
(Observe that \eqref{eq:uydef} in combination with $y \in Y_{ad} \cap K$ implies $y = S(u_y)$ so that 
we indeed do not have to mention this constraint explicitly here.)
By exploiting the formula \eqref{eq:lambdaydef} for the multiplier associated with $u_y$, we can go even further 
and recast \eqref{eq:Q} as an optimal control problem for the 
Poisson equation with state and control constraints as the following result shows:

\begin{corollary}[Reduction to an Optimal Control Problem for the Poisson Equation]
\label{cor:redux2}
A control $\bar u \in U_{ad}$ with associated state $\bar y := S(\bar u) \in Y_{ad}$
is a local (respectively, global) solution of \eqref{eq:Q} if and only if 
the function $\tilde u := \bar u + \mathds{1}_{\{\bar y = \psi\}}\max(0, -\Delta \psi)$
is a local (respectively, global) solution of the optimal control problem 
\begin{equation}
	\label{eq:StateConstrained}
	\begin{aligned}
		\text{Minimize} \quad & j(y) + \frac{\alpha}{2}\|u\|_{L^2}^2 +  \frac{\alpha}{2}\int_{\{y > \psi\}} \max(0, - \Delta \psi)^2 \mathrm{d}x \\
		\text{w.r.t.}\quad &(y,u) \in H_0^1(\Omega) \times L^2(\Omega) \\
		\text{s.t.}\quad & - \Delta y = u \\
		& u - \mathds{1}_{\{ y = \psi\}}\max(0, -\Delta \psi)\in U_{ad}  \\
		\text{and}\quad& y \in Y_{ad} \cap K.
	\end{aligned}
\end{equation}
\end{corollary}

\begin{proof}
Since our assumptions on $\Omega$ imply that there exists an absolute constant $C>0$
with $\|y\|_{H^2} \leq C \| \Delta y\|_{L^2}$ for all $y \in H_0^1(\Omega) \cap H^2(\Omega)$,
see \cite[Theorem 9.15, Lemma 9.17]{GilbargTrudinger2001} and \cite[Theorem 3.2.1.2]{Grisvard1985},
we obtain that $\tilde u \in L^2(\Omega)$ is a local (respectively, global) optimum 
of \eqref{eq:StateConstrained} if and only if the solution $\tilde y \in H_0^1(\Omega) \cap H^2(\Omega)$ 
of $-\Delta \tilde y = \tilde u$
is a local (respectively, global) optimum of  
\begin{equation}
\label{eq:random27363}
	\begin{aligned}
		\text{Minimize} \quad & j(y) + \frac{\alpha}{2}\|\Delta y\|_{L^2}^2 +  \frac{\alpha}{2}\int_{\{y > \psi\}} \max(0, - \Delta \psi)^2 \mathrm{d}x \\
		\text{w.r.t.}\quad &y \in  H_0^1(\Omega) \cap H^2(\Omega) \\
		\text{s.t.}\quad& -\Delta y - \mathds{1}_{\{ y = \psi\}}\max(0, -\Delta \psi)\in U_{ad}  \\
		\text{and}\quad& y \in Y_{ad} \cap K.
	\end{aligned}
\end{equation}
Here, with ``local'' we mean local w.r.t.\ the $L^2$-norm when referring to $\tilde u$ 
and local w.r.t.\ the $H^2$-norm when referring to $\tilde y$.
From Stampacchia's lemma,
we deduce that $-\Delta y = - \Delta \psi$ holds a.e.\ in $\{y = \psi\}$ for all $y \in Y_{ad}\cap K $.
If we exploit this identity, the definitions \eqref{eq:uydef} and \eqref{eq:lambdaydef},
and the fact that absolute constants are irrelevant for the minimization of the objective in \eqref{eq:random27363},
then we obtain that \eqref{eq:random27363} can also be written as
\begin{equation}
\label{eq:random273632}
	\begin{aligned}
		\text{Minimize} \quad & j(y) + \frac{\alpha}{2}\|u_y\|_{L^2}^2\\
		\text{w.r.t.}\quad &y \in  H_0^1(\Omega) \cap H^2(\Omega) \\
		\text{s.t.}\quad& u_y \in U_{ad}, \ \ y \in Y_{ad} \cap K.
	\end{aligned}
\end{equation}
The assertion for global optima is now a straightforward consequence of \cref{the:Qstructure} and \cref{cor:redux1}, cf.\ \eqref{eq:ReducedProblem}. It remains to prove the claim for local solutions. 
To this end, we again argue by contradiction:
Let us first assume that 
there exists $\bar u \in U_{ad}$ with state $\bar y := S(\bar u)$ 
such that $\bar u$ is a local solution of \eqref{eq:Q} 
and such that $\bar u + \mathds{1}_{\{\bar y = \psi\}}\max(0, -\Delta \psi)$
is not a local solution of \eqref{eq:StateConstrained}.
Then, 
\cref{the:Qstructure}
yields that  $\bar u$ and $\bar y$ satisfy $\bar u = u_{\bar y}$
and
 $-\Delta \bar y = \bar u + \mathds{1}_{\{\bar y = \psi\}}\max(0, -\Delta \psi)$,
the function $\bar u + \mathds{1}_{\{\bar y = \psi\}}\max(0, -\Delta \psi)$ is 
admissible for  \eqref{eq:StateConstrained},
and
we obtain from \eqref{eq:random27363} that we can find  a sequence $\{y_n\} \subset Y_{ad} \cap K$
with $u_{y_n} \in U_{ad}$, $y_n \to \bar y$ in $H^2(\Omega)$ and 
\begin{equation}
\label{eq:randomeq2836eg}
\begin{aligned}
& j(y_n) + \frac{\alpha}{2}\|\Delta y_n\|_{L^2}^2 +  \frac{\alpha}{2}\int_{\{y_n > \psi\}} \max(0, - \Delta \psi)^2 \mathrm{d}x
\\
&\quad <  
j(\bar y) + \frac{\alpha}{2}\|\Delta \bar y\|_{L^2}^2 +  \frac{\alpha}{2}\int_{\{\bar y > \psi\}} \max(0, - \Delta \psi)^2 \mathrm{d}x
\end{aligned}
\end{equation}
for all $n$. 
By taking the limes superior in \eqref{eq:randomeq2836eg}, we obtain 
\begin{equation}
\label{eq:randomeq362g62z3e}
\begin{aligned}
&j(\bar y) + \frac{\alpha}{2}\|\Delta \bar y\|_{L^2}^2 
+  \frac{\alpha}{2}\int_{\Omega} \mathds{1}_{\{\bar y > \psi\}} \max(0, - \Delta \psi)^2 \mathrm{d}x
\\
&\quad \geq \limsup_{n \to \infty}
j(y_n) + \frac{\alpha}{2}\|\Delta y_n\|_{L^2}^2 
+  \frac{\alpha}{2}\int_{\Omega} \mathds{1}_{\{y_n > \psi\}} \max(0, - \Delta \psi)^2 \mathrm{d}x
\\
&\quad \geq \liminf_{n \to \infty}
j(y_n) + \frac{\alpha}{2}\|\Delta y_n\|_{L^2}^2 
+  \frac{\alpha}{2}\int_{\Omega} \mathds{1}_{\{y_n > \psi\}} \max(0, - \Delta \psi)^2 \mathrm{d}x
\\
&\quad \geq \liminf_{n \to \infty}
j(y_n) + \frac{\alpha}{2}\|\Delta y_n\|_{L^2}^2 +  \frac{\alpha}{2}\int_{\Omega} \mathds{1}_{\{y_n > \psi, \ \bar y > \psi\}} \max(0, - \Delta \psi)^2 \mathrm{d}x
\\
&\quad =
j(\bar y) + \frac{\alpha}{2}\|\Delta \bar y\|_{L^2}^2 +  \frac{\alpha}{2}\int_{\Omega} \mathds{1}_{\{\bar y > \psi\}} \max(0, - \Delta \psi)^2 \mathrm{d}x,
\end{aligned}
\end{equation}
where the last equality follows from the dominated convergence theorem and $y_n \to \bar y$ in $H^2(\Omega)$, cf.\
the arguments in \eqref{eq:randomeq29738247}.
The above implies in particular that 
\begin{equation*}
\int_{\Omega} \mathds{1}_{\{y_n > \psi\}} \max(0, - \Delta \psi)^2 \mathrm{d}x \to \int_{\Omega} \mathds{1}_{\{\bar y > \psi\}} \max(0, - \Delta \psi)^2 \mathrm{d}x
\end{equation*}
and, as a consequence, that
\begin{equation*}
\begin{aligned}
&\int_{\Omega} \left ( \mathds{1}_{\{y_n > \psi\}} \max(0, - \Delta \psi) - \mathds{1}_{\{\bar y > \psi\}} \max(0, - \Delta \psi) \right )^2 \mathrm{d}x
\\
&\quad = 
\int_{\Omega}  
\max(0, - \Delta \psi)^2 \left ( \mathds{1}_{\{y_n > \psi\}} - 2 \mathds{1}_{\{\bar y > \psi, y_n > \psi\}}  +  \mathds{1}_{\{\bar y > \psi\}} \right )  \mathrm{d}x \to 0
\end{aligned}
\end{equation*}
for $n \to \infty$. Thus,
$
 \mathds{1}_{\{y_n > \psi\}} \max(0, - \Delta \psi) \to \mathds{1}_{\{\bar y > \psi\}} \max(0, - \Delta \psi) 
$
in $L^2(\Omega)$ for $n \to \infty$, and we may use \eqref{eq:lambdaydef} to deduce that 
the multipliers $\lambda_{\bar y}$ and $\lambda_{y_n}$ associated with the controls $\bar u = u_{\bar y}$ and $u_{y_n}$ satisfy
$
\lambda_{y_n} \to \lambda_{\bar y}
$
for $n \to \infty$ in $L^2(\Omega)$. 
Since the sequence $\{y_n\}$ satisfies $\Delta y_n \to \Delta \bar y $ in $L^2(\Omega)$
by its construction, the convergence $\lambda_{y_n} \to \lambda_{\bar y}$ in $L^2(\Omega)$ yields  
$u_{y_n} = -\Delta y_n - \lambda_{y_n} \to -\Delta \bar y- \lambda_{\bar y} = \bar u$ in $L^2(\Omega)$.
From \eqref{eq:randomeq2836eg}, we may now deduce that $\bar u$ does not satisfy
an inequality of the form \eqref{eq:locoptimality2}, and from \cref{cor:redux1}
that $\bar u$ cannot be locally optimal for \eqref{eq:Q}. This is a contradiction. 
For a local solution $\bar u$ of \eqref{eq:Q}, the function $\tilde u := \bar u + \mathds{1}_{\{\bar y = \psi\}}\max(0, -\Delta \psi)$
is thus always a local solution of \eqref{eq:StateConstrained}.

To prove the reverse implication, we can proceed along similar lines:
Let us assume  that there exists a local minimum
$\tilde u$ of \eqref{eq:StateConstrained}
such that $\bar u := \tilde u - \mathds{1}_{\{\bar y = \psi\}}\max(0, -\Delta \psi)$
is not a local minimum of \eqref{eq:Q}, where $\bar y$ denotes the solution of $-\Delta \bar y = \tilde u$.
Then, it follows from the conditions in \eqref{eq:StateConstrained}, the formulas \eqref{eq:uydef} and \eqref{eq:lambdaydef},
and \cref{cor:redux1}
that $\bar u = u_{\bar y}$ and $\bar y = S(\bar u)$ holds, that $\bar u$ is admissible for \eqref{eq:ReducedProblem},
and that there exists a sequence $\{y_n\} \subset Y_{ad} \cap K$ with $u_{y_n} \in U_{ad}$ for all $n$, 
$u_{y_n} \to \bar u$ in $L^2(\Omega)$ for $n \to \infty$ and 
\begin{equation}
\label{eq:randomeq1836623}
 j(y_n) + \frac{\alpha}{2}\|u_{y_n}\|_{L^2}^2
<
 j(\bar y) + \frac{\alpha}{2}\|\bar u\|_{L^2}^2
\end{equation}
for all $n$. 
Note that the convergence $u_{y_n} \to \bar u$ in $L^2(\Omega)$, the identity $y_n = S(u_{y_n})$, the estimate \eqref{eq:H2reg} 
and \cref{th:solutionmapproperties} imply that $y_n$ has to converge weakly in $H^2(\Omega)$ and strongly in $H^1(\Omega)$ to $\bar y$.
By rewriting \eqref{eq:randomeq1836623} analogously to \eqref{eq:random27363},
by taking the limes superior, by exploiting the weak lower semicontinuity of 
continuous and convex functions, and by using the same arguments as in \eqref{eq:randomeq362g62z3e},
we now obtain 
\begin{equation*}
\begin{aligned}
&j(\bar y) + \frac{\alpha}{2}\|\Delta \bar y\|_{L^2}^2 
+  \frac{\alpha}{2}\int_{\Omega} \mathds{1}_{\{\bar y > \psi\}} \max(0, - \Delta \psi)^2 \mathrm{d}x
\\
&\quad \geq \limsup_{n \to \infty}
j(y_n) + \frac{\alpha}{2}\|\Delta y_n\|_{L^2}^2 
+  \frac{\alpha}{2}\int_{\Omega} \mathds{1}_{\{y_n > \psi\}} \max(0, - \Delta \psi)^2 \mathrm{d}x
\\
&\quad \geq 
j(\bar y)  +  \frac{\alpha}{2}\int_{\Omega} \mathds{1}_{\{\bar y > \psi\}} \max(0, - \Delta \psi)^2 \mathrm{d}x
+
\limsup_{n \to \infty}
 \frac{\alpha}{2}\|\Delta y_n\|_{L^2}^2 
\\
&\quad \geq 
j(\bar y)  +  \frac{\alpha}{2}\int_{\Omega} \mathds{1}_{\{\bar y > \psi\}} \max(0, - \Delta \psi)^2 \mathrm{d}x
+
\liminf_{n \to \infty}
 \frac{\alpha}{2}\|\Delta y_n\|_{L^2}^2 
\\
&\quad \geq
j(\bar y) + \frac{\alpha}{2}\|\Delta \bar y\|_{L^2}^2 +  \frac{\alpha}{2}\int_{\Omega} \mathds{1}_{\{\bar y > \psi\}} \max(0, - \Delta \psi)^2 \mathrm{d}x.
\end{aligned}
\end{equation*}
The above implies $\Delta y_n \to \Delta \bar y$ in $L^2(\Omega)$
and, again by the estimate  $\|y\|_{H^2} \leq C \| \Delta y\|_{L^2}$ for all $y \in H_0^1(\Omega) \cap H^2(\Omega)$,
that $y_n \to \bar y$ in $H^2(\Omega)$. The inequality \eqref{eq:randomeq1836623} now yields
that $\bar y$ cannot be a local optimum of \eqref{eq:random273632} and, by 
the considerations at the beginning of this proof, that $\bar y$ cannot be locally optimal for \eqref{eq:random27363}
and that $\tilde u$ cannot be locally optimal for \eqref{eq:StateConstrained}.
This again contradicts our assumptions and completes the proof. 
\end{proof}

Several things are noteworthy regarding the last result:

\begin{remark}~
\begin{itemize}
\item In the literature, optimal control problems with state constraints and 
optimal control problems governed by obstacle-type variational inequalities 
are typically treated as two different problem classes, cf.\ the discussion in 
\cite[Section 1]{KunischWachsmuth2012}.
\cref{cor:redux2} shows that this distinction is, in fact, not entirely appropriate
since it is perfectly possible to restate an optimal control problem of the form \eqref{eq:Q}
as a state- and control-constrained optimal control problem for the Poisson equation (albeit with a modified objective function).

\item We would like to point out that the reformulation \eqref{eq:StateConstrained} 
of the problem \eqref{eq:Q} implies that
it is energetically favorable for a tuple $(y, u)$  
to have a large contact set $\{y=\psi\}$ in those parts of the domain $\Omega$
where the Laplacian $\Delta \psi$ is negative. This also makes sense in view 
of formula \eqref{eq:uydef} which yields that
the partially optimal control $u_y$ vanishes a.e.\ in the set $\{y = \psi,\, \Delta \psi < 0\}$.
A similar behavior is not present when, e.g., a state-constrained 
tracking-type optimal control problem governed by the Poisson equation is considered.  
\end{itemize}
\end{remark}

An important observation at this point is that both the additional term in the objective function
 of \eqref{eq:StateConstrained} and 
the right-hand side of  \eqref{eq:lambdaydef}  only depend on the negative part of the Laplacian $\Delta \psi$.
If $\Delta \psi \geq 0$ holds a.e.\ in $\Omega$, then $\lambda_y$ is identical zero for all states $y$,
the objective of \eqref{eq:StateConstrained} is identical to that of \eqref{eq:Q},
and the analysis simplifies drastically as the following section shows.
 
\section{Enhanced Second-Order Conditions, Global Optimality and Quadratic Growth for Subharmonic Obstacles}
\label{sec:5}

In the special case of a subharmonic obstacle, i.e., if $\Delta \psi \geq 0$ holds a.e.\ in $\Omega$,
the findings of \cref{sec:4} give rise to the following, quite remarkable result:

\begin{theorem}[Reformulation of Problems with Subharmonic Obstacles]
\label{th:reduxsubharmonic}
Suppose that the function $\psi$ satisfies $\Delta \psi \geq 0$ a.e.\ in $\Omega$.
Then, \eqref{eq:Q} is equivalent to the control- and state-constrained 
optimal control problem 
\begin{equation}
	\label{eq:StateConstrainedSubharmonic}
	\begin{aligned}
		\text{Minimize} \quad & j(y) + \frac{\alpha}{2}\|u\|_{L^2}^2  \\
		\text{w.r.t.}\quad &(y,u) \in H_0^1(\Omega) \times L^2(\Omega) \\
		\text{s.t.}\quad & - \Delta y = u \\
		\text{and}\quad& u \in U_{ad},\quad y \in Y_{ad} \cap K
	\end{aligned}
\end{equation}
in the following sense:
\begin{enumerate}
\item Every local (respectively, global) solution $\bar u$ of \eqref{eq:Q} is a local (respectively, global) solution 
of \eqref{eq:StateConstrainedSubharmonic} and vice versa.
\item A point $\bar u \in U_{ad}$ with associated state $\bar y := S(\bar u) \in Y_{ad}$ 
satisfies a local quadratic growth condition of the form \eqref{eq:quadgrowth1}
with constants $r, c > 0$ for the problem \eqref{eq:Q} if and only if 
an analogous local quadratic growth condition (with possibly different constants) holds for \eqref{eq:StateConstrainedSubharmonic}.
\item A point $\bar u \in U_{ad}$ with associated state $\bar y := S(\bar u) \in Y_{ad}$   satisfies a global quadratic growth condition for  \eqref{eq:Q}
(i.e., an inequality of the form \eqref{eq:quadgrowth1} with a constant $c>0$ and $r = \infty$) 
if and only if an analogous global quadratic growth condition (with a possibly different constant $c$) holds for \eqref{eq:StateConstrainedSubharmonic}.
\end{enumerate}
\end{theorem}

\begin{proof}
From the non-negativity of the Laplacian $\Delta \psi$ a.e.\ in $\Omega$
and the formulas \eqref{eq:uydef} and \eqref{eq:lambdaydef}, we obtain that 
the partially optimal controls $u_y$ and the associated multipliers $\lambda_y$ satisfy
$u_y = - \Delta y$ and $\lambda_y = 0$ for all attainable states $y$. This implies in 
particular that the problems \eqref{eq:ReducedProblem}  and \eqref{eq:StateConstrained}
take precisely the form \eqref{eq:StateConstrainedSubharmonic}.
The claims of the theorem now follow immediately from \cref{the:Qstructure,cor:redux1,cor:redux2}. 
\end{proof}

As \cref{th:reduxsubharmonic} shows, under the assumption of subharmonicity,
the optimal control problem \eqref{eq:Q} for the 
obstacle problem and the optimal control problem \eqref{eq:StateConstrainedSubharmonic} for the Poisson equation
are fully equivalent in terms of local/global optima and local/global quadratic growth properties. 
In particular, we may conclude that every condition that is necessary/sufficient
for local/global optimality or local/global quadratic growth in \eqref{eq:StateConstrainedSubharmonic}
is also necessary/sufficient 
for local/global optimality or local/global quadratic growth in \eqref{eq:Q} and vice versa. 
By exploiting this observation, we obtain, e.g., the following result: 

\begin{corollary}[SSC for Local Optimality in the Presence of Subharmonicity]
\label{cor:subharmonicSSC1}
Suppose that $\psi$ satisfies $\Delta \psi \geq 0$ a.e.\ in $\Omega$ and that $Y_{ad}$ is convex. Assume further that
a control
$\bar u \in U_{ad}$ with state $\bar y := S(\bar u) \in Y_{ad}$ is given such that $\bar u$
satisfies the Bouligand stationarity condition \eqref{eq:BouligandstatcondQ} of \eqref{eq:Q} and such that 
\begin{equation}
\label{eq:anSSC}
j''(\bar y)T(h)^2 + \alpha \|h\|_{L^2}^2 >  0 \quad \forall h \in 
\TT_{\UU_{ad}^{biact}}(\bar u) \setminus\{0\}\text{ with } \left \langle j'(\bar y), T(h) \right \rangle + \alpha \left ( \bar u, h \right )_{L^2} = 0
\end{equation}
holds, 
where $T: H^{-1}(\Omega) \to H_0^1(\Omega)$ denotes the solution map of the Poisson equation, 
where $\UU_{ad}^{biact}$ denotes the set of all controls of the effective admissible set $\UU_{ad}$ of \eqref{eq:Q}
whose states have a vanishing multiplier, i.e.,  $\UU_{ad}^{biact} = U_{ad} \cap T^{-1}(K \cap Y_{ad})$,
and where $\TT_{\UU_{ad}^{biact}}(\bar u)$ denotes the tangent cone of $\smash{\UU_{ad}^{biact}}$ at $\bar u$. Then, 
there exist constants $r, c > 0$ such that $\bar u$ satisfies a local quadratic growth condition of the form \eqref{eq:quadgrowth1}
for \eqref{eq:Q}.
\end{corollary}

\begin{proof}
From \eqref{eq:BouligandstatcondQ}, the sign of $\Delta \psi$ and \cref{the:Qstructure}, 
we obtain that $-\Delta \bar y = \bar u$ holds and that $\bar u$ satisfies 
\begin{equation}
\label{eq:BouligandPoisson}
\left \langle j'(\bar y), T(h) \right \rangle + \alpha \left ( \bar u, h \right )_{L^2} \geq 0 \qquad 
\forall h \in 
\TT_{\UU_{ad}^{biact}}(\bar u).
\end{equation}
The above implies in particular that $\bar u$ is Bouligand stationary for \eqref{eq:StateConstrainedSubharmonic},
and that we may invoke \cite[Lemma 3.2ii), Theorem 4.4, Lemma 5.1]{ChristofWachsmuth2018}
to deduce that \eqref{eq:anSSC} is a sufficient condition for local quadratic 
growth in \eqref{eq:StateConstrainedSubharmonic}. 
(Note that the admissible set $U_{ad} \cap T^{-1}(K \cap Y_{ad})$ of \eqref{eq:StateConstrainedSubharmonic} is trivially convex.)
The claim is now a straightforward consequence of \cref{th:reduxsubharmonic}.
\end{proof}

If we additionally assume that $j$ is convex,
then we obtain:

\begin{corollary}[Unique Solvability for Problems with Subharmonic Obstacles]
\label{cor:subharmonicSSC2}
Suppose that $\psi$ satisfies $\Delta \psi \geq 0$ a.e.\ in $\Omega$, that $Y_{ad}$ is convex,
and that $j$ is convex. Then, \eqref{eq:Q} admits one and only one local/global 
solution $\bar u$, this solution is uniquely determined by \eqref{eq:BouligandstatcondQ},
and there exists a constant $c>0$ such that $\bar u$ satisfies a global quadratic growth condition
(i.e., a condition of the form \eqref{eq:quadgrowth1} with $r=\infty$).
\end{corollary}

\begin{proof}
From the convexity of $Y_{ad}$ and $j$, it follows that the objective function of \eqref{eq:StateConstrainedSubharmonic}
is strongly convex and that \eqref{eq:StateConstrainedSubharmonic} is a convex problem. 
This implies in particular that \eqref{eq:StateConstrainedSubharmonic} admits one and only one local/global solution 
which is uniquely determined by the Bouligand stationarity condition \eqref{eq:BouligandPoisson} of \eqref{eq:StateConstrainedSubharmonic}
and which satisfies a global quadratic growth condition.
The claim now follows immediately from \cref{th:reduxsubharmonic}, cf.\ also the proof of \cref{cor:subharmonicSSC1}.
\end{proof}

We would like to point out that, even for a subharmonic $\psi$ and convex  $j$ and $Y_{ad}$, 
it is typically completely unclear whether \eqref{eq:Q} is a convex minimization problem or not. 
To the authors' knowledge, 
the convexity of \eqref{eq:Q} could be established so far only for the quite pathological case of 
a classical tracking-type optimal control problem with a desired state $y_D$ 
satisfying $y_D \leq \psi$ a.e.\ in $\Omega$, see
\cite[Théorème~4.1]{Mignot1976}.
Our analysis shows, however, that all points that could possibly prevent \eqref{eq:Q} from being convex 
are suboptimal in the situation of \cref{cor:subharmonicSSC2}. 
Because of this effect, \eqref{eq:Q} effectively behaves like a convex problem and we are able to prove the 
uniqueness of its solution. 
For the sake of completeness, 
we also state the following corollary for the problem \eqref{eq:P}:

\begin{corollary}[SSC for Local Optimality in \eqref{eq:P} in the Presence of Subharmonicity]
\label{cor:subharmonicSSC3}
Suppose that $\psi$ satisfies $\Delta \psi \geq 0$ a.e.\ in $\Omega$, 
and that a control $\bar u \in U_{ad}$ is given which satisfies the strong stationarity system \eqref{eq:strongstationarity} of \eqref{eq:P}
with a triple $(\bar p, \bar \nu, \bar \eta) \in H_0^1(\Omega) \times L^2(\Omega) \times H^{-1}(\Omega)$,
state $\bar y := S(\bar u) \in H_0^1(\Omega)$ and multiplier $\bar \lambda := -\Delta \bar y - \bar u \in L^2(\Omega)$. 
Assume further that
\begin{equation*}
j''(\bar y)T(h)^2 + \alpha \|h\|_{L^2}^2 >  0\qquad \forall h \in 
\bar \nu^\perp \cap T^{-1}(\bar \eta^\perp) \cap \TT_{U_{ad}^{biact}}(\bar u) \setminus \{0\}
\end{equation*}
holds, where $T: H^{-1}(\Omega) \to H_0^1(\Omega)$
again denotes the solution map  of the Poisson problem, and where 
$U_{ad}^{biact} := U_{ad} \cap T^{-1}(K)$. Then,
$\bar u$ is locally optimal for \eqref{eq:P} and satisfies a local quadratic growth condition 
of the form \eqref{eq:quadgrowth1} with some constants $c,r > 0$. If, additionally, the function $j$ is convex,
then there exists at most one $\bar u$ which satisfies the strong stationarity condition \eqref{eq:strongstationarity},
and \eqref{eq:strongstationarity} is a sufficient condition for global optimality and 
global quadratic growth. 
\end{corollary}

\begin{proof}
The claim follows immediately from \eqref{eqrandomeq82636} and \cref{cor:subharmonicSSC1,cor:subharmonicSSC2}.
\end{proof}

Note that we could also state necessary second-order optimality conditions for problems \eqref{eq:Q} 
with subharmonic obstacles at this point by
proceeding completely analogously to the proofs of \cref{cor:subharmonicSSC1,cor:subharmonicSSC2,cor:subharmonicSSC3} 
and by invoking corresponding results for (special instances of) \eqref{eq:StateConstrainedSubharmonic}
as found, e.g., in \cite{NhuSonYao2016} and \cite{BonnansNecSec2009}. For the sake of brevity, we do not go into the details here.

Before we turn our attention to problems with general obstacles, 
we would like to mention that the equivalence in \cref{th:reduxsubharmonic} is also interesting for the analysis 
of control- and state-constrained optimal control problems of the form \eqref{eq:StateConstrainedSubharmonic}. 
In combination with the results of \cref{sec:2}, for example, 
\cref{th:reduxsubharmonic} yields that the subharmonicity of the bound $\psi$ in \eqref{eq:StateConstrainedSubharmonic} 
can be used as a constraint qualification that ensures the existence of a multiplier system 
even in the absence of Slater points:

\begin{corollary}[Multipliers for State-Constrained Problems without Slater Points]
\label{cor:strongasKKT}
For every local solution $\bar u$
of an optimal control problem of the form \eqref{eq:StateConstrainedSubharmonic} 
that satisfies $Y_{ad} = H^2(\Omega)$, $\Delta \psi \geq 0$ a.e.\ in $\Omega$, and 
one of the conditions \ref{item:strong-i}
and \ref{item:strong-iii} in \cref{th:strongstationaritynecessary}, 
there exist an adjoint state $\bar p \in H_0^1(\Omega)$
and multipliers $\bar \nu \in L^2(\Omega)$, $\bar \eta \in  H^{-1}(\Omega)$
such that  $\bar u$, its state $\bar y$, $\bar p$, $\bar \nu$, and $\bar \eta$ satisfy the system
\eqref{eq:strongstationarity}
with $\bar\lambda = 0$.
\end{corollary}

\begin{proof}
The claim is a straightforward consequence of the equivalence in \cref{th:reduxsubharmonic} and
the necessity of the strong stationarity system in  \cref{th:strongstationaritynecessary}.
\end{proof}

Note that \eqref{eq:strongstationarity} implies in particular that the adjoint state $\bar p$ and the optimal control $\bar u$
enjoy $H_0^1(\Omega)$-regularity in the situation of \cref{cor:strongasKKT}. 
Normally, one would only obtain $\bar p, \bar u \in W_0^{1,s}(\Omega)$ here for all $1 \leq s < d/(d-1)$,
cf.\ \cite[Proposition 1]{KunischBergounioux2002} and \cite[Theorem 2.1]{CasasMateosVexler2014}.
We remark that this higher regularity of $\bar p$ and $\bar u$ for problems of the type \eqref{eq:StateConstrainedSubharmonic}
with $Y_{ad} = H^2(\Omega)$
has already been proved under different assumptions on $\psi$ and in the presence of a Slater point in
\cite[Theorem 3.1]{CasasMateosVexler2014}
by exploiting properties of the Green's function associated with the Poisson equation $-\Delta y = u$. 
We obtain the same result along completely different lines, namely, by including the 
state constraint into the solution operator and by utilizing the 
stability properties of the obstacle problem \eqref{eq:obstacleproblem}. 

\section{Enhanced Second-Order Conditions for General Obstacles}
\label{sec:6}

For non-subharmonic $\psi$, the additional terms in the objective function and the constraints of \eqref{eq:StateConstrained} 
cannot be neglected and the derivation of second-order optimality conditions naturally becomes more complicated.
In what follows, we will show that it is nevertheless possible to 
improve the results collected in \cref{th:SSCmajor}  for general obstacles $\psi$ 
by exploiting 
the observations made in \cref{sec:4}. 
To simplify the analysis, henceforth, we again consider the problem \eqref{eq:P}, i.e., 
we restrict our attention to the case $Y_{ad} = H^2(\Omega)$.
The main result of this section is the following: 

\begin{theorem}
\label{th:enhancedSSC}
Suppose that $\max(0, -\Delta \psi) \leq u_b$ holds a.e.\ in $\Omega$,
and that a control $\bar u \in U_{ad}$ is given which satisfies the strong stationarity system \eqref{eq:strongstationarity} of the problem 
\eqref{eq:P} with a triple $(\bar p, \bar \nu, \bar \eta) \in H_0^1(\Omega) \times L^2(\Omega) \times H^{-1}(\Omega)$,
state $\bar y := S(\bar u) \in H_0^1(\Omega)$ and multiplier $\bar \lambda \in L^2(\Omega)$.
Then, the following is true:
\begin{enumerate}
\item \label{item:enhancedSSC-i} 
If 
there exist constants $\beta \geq 0$ and $\gamma,\delta> 0$ with
\begin{gather}
	-\alpha \bar u + \beta (\bar y - \psi) \geq 0 \text{ a.e.\ in } \{0 < \bar y - \psi < \gamma, \,\Delta \psi < 0, \, 0 < \bar u < -(2+\delta)\Delta \psi\},
\label{eq:randomeq242315-1}
\\
\bar \eta + \beta \mathds{1}_{\{\bar y = \psi\}}\max(0, - \Delta \psi) \geq 0 \text{ in the sense of }H^{-1}(\Omega),
\label{eq:randomeq242315-2}
\end{gather}
and if 
\begin{equation}
\label{eq:randomeq2423152}
j''(\bar y)S'(\bar u; h)^2 + \alpha \|h\|_{L^2}^2 >  0
\end{equation}
holds for all $h \in \TT_{U_{ad}}(\bar u) \setminus \{0\}$ with $S'(\bar u; h) \in \bar \eta^\perp$,
then $\bar u$ is locally optimal for \eqref{eq:P} and there exist constants $c, \varepsilon > 0$ with
\begin{equation}
\label{eq:randomgrowth22}
J(S(u), u) \geq J(S(\bar u), \bar u) + \frac{c}{2} \|u - \bar u\|_{L^2}^2\qquad \forall u \in U_{ad} \cap B_\varepsilon^{L^2}(\bar u).
\end{equation}

\item \label{item:enhancedSSC-ii} 
If $\bar u$ satisfies 
\begin{equation}
\label{eq:globalcondition72623}
\bar u \not \in \big ( 0, -2\Delta \psi  \big ) \text{ a.e.\ in } \{\bar y > \psi, \Delta \psi < 0 \},
\end{equation}
and if there exist constants $\beta \geq 0$ and $\mu \geq 0$ such that
\begin{equation}
\label{eq:randomeq827373}
\begin{gathered}
\bar \eta + \beta \mathds{1}_{\{\bar y = \psi\}}\max(0, - \Delta \psi) \geq 0 \text{ in the sense of }H^{-1}(\Omega),
\\
j''(y)z^2 \geq \mu \|z\|_{L^2}^2 \quad \forall (y, z) \in K \times (K - \bar y),
\end{gathered}
\end{equation}
and
\begin{equation}
\label{eq:globalinequality22}
 \mu + 2 \beta \omega - \frac{\beta^2 }{ \alpha} \geq 0
\end{equation}
holds, where the Poincaré constant $\omega$ is again defined by \eqref{eq:PoincareDef}, then $\bar u$ is globally optimal for \eqref{eq:P}. If, further,
$\mu$ is positive and \eqref{eq:globalinequality22} is strict, 
then $\bar u$ is even the unique global optimum of the problem \eqref{eq:P}.
\end{enumerate}
\end{theorem}

\begin{proof}
The proof of  \cref{th:enhancedSSC} is along the lines of that of \cref{th:SSCmajor} and again based 
on contradiction arguments and the  expansion \eqref{eq:Taylorexpansion}. 

To prepare some of the subsequent steps,
consider an arbitrary state $y \in H_0^1(\Omega) \cap H^2(\Omega)$
that is attainable in \eqref{eq:P}, and denote with $u_y \in U_{ad}$ and $\lambda_y$ the 
partially optimal control of $y$ and the associated multiplier as in \eqref{eq:lambdaydef}, respectively. 
Further,
we define the sets
\begin{align*}
	A_y &:= \{ \bar y > \psi, \; y = \psi, \; \Delta \psi < 0 \}
	,
	&
	B_y &:= A_y \cap \{ 0 < \bar u < -(2+\delta) \, \Delta \psi \}
	.
\end{align*}
Here, we use $\delta= 0$ in case \ref{item:enhancedSSC-ii}.

Note that $\bar \lambda = -\Delta\psi$ a.e.\ on $\{\bar y = \psi, \; \Delta\psi < 0\}$.
Thus, $\bar p = 0$ a.e.\ on this set.
Now,
using \eqref{eq:strongstationarity_2}, \eqref{eq:strongstationarity_3} and 
\eqref{eq:lambdaydef},
we have
\begin{equation}
\label{eq:randomeq286328}
	\dual{\bar p}{\lambda_y - \bar \lambda}
	= \dual{\bar p}{\lambda_y}
	=
	- \int_{\{y = \psi, \; \Delta\psi < 0\}} \bar p \, \Delta\psi \,\dx
	= \int_{A_y} (\alpha \bar u - \bar \nu) \Delta\psi \,\dx
	.
\end{equation}
Next, we check that
$u_y - \mathds{1}_{A_y} \Delta\psi \in U_{ad}$.
Indeed, from \eqref{eq:uydef}, we obtain
\begin{equation*}
	(u_y - \mathds{1}_{A_y} \Delta\psi)(x)
	=
	0 - \Delta\psi(x) \in [0,u_b(x)]
\end{equation*}
for a.a.\ $x \in A_y$.
This admissibility implies $(\bar\nu, u_y - \mathds{1}_{A_y}\Delta\psi - \bar u)_{L^2} \ge 0$.
Combining the last inequality with \eqref{eq:randomeq286328}
yields
\begin{align*}
	&\dual{\bar p}{\lambda_y - \bar\lambda}
	+
	(\bar\nu, u_y - \bar u)_{L^2}
	+
	\frac\alpha2 \norm{u_y - \bar u}_{L^2}^2
	\\&\qquad=
	\alpha \int_{A_y} \bar u \Delta\psi \,\dx
	+
	(\bar\nu, u_y - \mathds{1}_{A_y} \Delta\psi - \bar u)_{L^2}
	+
	\frac\alpha2 \norm{u_y - \bar u}_{L^2}^2
	\\
	&\qquad\ge
	\alpha \int_{A_y} \bar u \Delta\psi + \frac1{2+\delta} \bar u^2 \,\dx
	+
	\frac\alpha2 \norm{u_y - \bar u}_{L^2(\Omega\setminus A_y)}^2
	+
	\frac{\delta\,\alpha}{2\,(2+\delta)} \norm{u_y - \bar u}_{L^2(A_y)}^2
	.
\end{align*}
Finally,
we use that $\bar u (\Delta \psi + \frac1{2+\delta} \bar u) \ge 0$
on $\{\bar u \not\in (0,-(2+\delta)\Delta\psi), \; \Delta\psi < 0\}$.
Hence,
\begin{equation}
	\label{eq:sadroeiiaoedr}
	\begin{aligned}
		&
		\dual{\bar p}{\lambda_y - \bar\lambda}
		+
		(\bar\nu, u_y - \bar u)_{L^2}
		+
		\frac\alpha2 \norm{u_y - \bar u}_{L^2}^2
		\\&\qquad
		\ge
		\alpha \int_{B_y} \bar u \Delta\psi \,\dx
		+
		\frac\alpha2 \norm{u_y - \bar u}_{L^2(\Omega\setminus A_y)}^2
		+
		\frac{\delta\,\alpha}{2\,(2+\delta)} \norm{u_y - \bar u}_{L^2(A_y)}^2
		.
	\end{aligned}
\end{equation}

We are now in the position to verify \ref{item:enhancedSSC-i}: 
Suppose that we are given a strongly stationary point $\bar u \in U_{ad}$ with associated $\bar y$, $\bar \lambda$, $\bar p$,
$\bar \nu$ and $\bar \eta$ such that
the conditions in \eqref{eq:randomeq242315-1}, \eqref{eq:randomeq242315-2}  and \eqref{eq:randomeq2423152} are satisfied 
and such that \eqref{eq:randomgrowth22} is violated. Then, 
it follows from \cref{the:Qstructure} and the fact that strong stationarity implies 
Bouligand stationarity that $\bar u = u_{\bar y}$ holds, where $ u_{\bar y}$ is again defined by \eqref{eq:uydef},
and we may invoke \cref{cor:redux1} to deduce that there exist sequences 
$\{y_n\} \subset S(U_{ad})$ and  $\{c_n\} \subset \mathbb{R}^+$
such that $\{y_n\}$, $\{c_n\}$ and the controls $\{u_{y_n}\} \subset U_{ad}$ defined in \eqref{eq:uydef}
satisfy  
	\begin{equation*}
		u_{y_n} \in U_{ad},\quad
		c_n \searrow 0
		,\quad
		\norm{u_{y_n} - \bar u }_{L^2} \to 0
		\quad\text{and}\quad
		J(y_n, u_{y_n}) - J(\bar y, \bar u) < \frac{c_n}{2} \|u_{y_n} - \bar u\|_{L^2}^2.
	\end{equation*} 
Define $t_n := \norm{u_{y_n} - \bar u}_{L^2}$, $h_n := (u_{y_n} - \bar u) / t_n$, and denote
the multipliers associated with $y_n$ and $u_{y_n}$  in \eqref{eq:lambdaydef} with
$\lambda_{y_n}$.
Then, it holds $t_n \searrow 0$, $\|h_n\|_{L^2}=1$, and we may again assume w.l.o.g.\ that the sequence $h_n$ converges weakly in $L^2(\Omega)$ to some 
$h \in \TT_{U_{ad}}(\bar u)$ for $n \to \infty$. Using \eqref{eq:Taylorexpansion}, the continuity of $j''$, 
and 
the fact that $h_n \weakly h$ in $L^2(\Omega)$ implies $(y_n - \bar y)/t_n \to S'(\bar u; h)$ in $H_0^1(\Omega)$,
we may now deduce that 
\begin{equation}
\label{eq:randomeq2816363}
\begin{aligned}
		0
		&\geq
		\frac{J( y_n, u_{y_n}) - J(\bar y, \bar u) - \frac{c_n}{2} \|t_n h_n \|_{L^2}^2}{t_n^2}
		\\
		& = \frac{1}{t_n^2} \big (
	\left \langle  \bar p  , \lambda_{y_n} - \bar \lambda \right \rangle +
	\left \langle  \bar \eta, y_n  - \bar y \right \rangle + ( \bar \nu,  u_{y_n} - \bar u )_{L^2} \big )
	+ \frac{1}{2} j''(\bar y) S'(\bar u; h)^2  + \frac{\alpha}{2} \| h_n\|_{L^2}^2 + \oo(1),
\end{aligned}
\end{equation}
where the Landau symbol refers to the limit $n \to \infty$. 
Using additionally \eqref{eq:sadroeiiaoedr} with $y = y_n$,
we find
\begin{equation*}
	0
	\ge
	\frac1{t_n^2}
	\dual{\bar\eta}{y_n - \bar y}
	+
	\frac\alpha{t_n^2} \int_{B_{y_n}} \bar u \Delta\psi \,\dx
	+
	\zeta_n + \oo(1)
	,
\end{equation*}
where we used the abbreviation
\begin{equation*}
	\zeta_n
	:=
	\frac{1}{2} j''(\bar y) S'(\bar u; h)^2
	+
	\frac\alpha2 \norm{h_n}_{L^2(\Omega\setminus A_{y_n})}^2
	+
	\frac{\delta\,\alpha}{2\,(2+\delta)} \norm{h_n}_{L^2(A_{y_n})}^2
	.
\end{equation*}
By exactly the same arguments as in the proof of part \ref{item:MajorSSC-i} of  \cref{th:SSCmajor}, 
we obtain that $\{y_n = \psi, \bar y > \psi\} \subset \{0 < \bar y - \psi \leq C t_n\}$ holds with an absolute constant $C>0$.
Thus,
\eqref{eq:randomeq242315-1} implies
that
$-\alpha\bar u + \beta\,(\bar y-\psi) \ge 0$
holds a.e.\ on $B_{y_n}$ for $n$ large enough, and 
we arrive at
\begin{align*}
	0
	&\ge
	\frac1{t_n^2} \dual{\bar\eta}{y_n - \bar y}
	+                
	\frac\beta{t_n^2} \int_{B_{y_n}} (\bar y - \psi) \, \Delta \psi \,\dx
	+
	\zeta_n + \oo(1)
	\\&\ge
	\frac1{t_n^2} \dual{\bar\eta}{y_n - \bar y}
	+                
	\frac\beta{t_n^2} \int_{A_{y_n}} (\bar y - \psi) \, \Delta \psi \,\dx
	+
	\zeta_n + \oo(1)
	.
\end{align*}
Now, we can use that $\psi = y_n$
and $\lambda_{y_n} = -\Delta\psi$ a.e.\ on $A_{y_n}$
as well as
$(\bar y - y_n)\lambda_{y_n} \ge 0$
a.e.\ on $\Omega$
to obtain
\begin{equation*}
	0 \ge
	\frac1{t_n^2} \dual{\bar\eta}{ y_n - \bar y }
	-
	\frac\beta{t_n^2} \int_{\Omega} (\bar y - y_n) \, \lambda_{y_n} \,\dx
	+
	\zeta_n + \oo(1)
	.
\end{equation*}
Recall that, due to
the convergence $(y_n - \bar y)/t_n \to S'(\bar u; h)$ in $H_0^1(\Omega)$ and the variational inequality \eqref{eq:VIdirdiff}, 
we have
\begin{equation*}
\left \langle \frac{\lambda_{y_n} - \bar \lambda}{t_n}, \frac{y_n - \bar y}{t_n}  \right \rangle 
\to
 \left \langle -\Delta S'(\bar u; h) - h,  S'(\bar u; h) \right \rangle
= 0. 
\end{equation*}
Thus,
\begin{equation}
\label{eq:draenoidreaoi}
	0 \ge
	\frac1{t_n^2} \dual{\bar\eta + \beta\bar\lambda}{y_n - \bar y}
	+
	\zeta_n + \oo(1)
	.
\end{equation}
Using \eqref{eq:aux_in_acht_1}, \eqref{eq:aux_in_acht_3}, \eqref{eq:randomeq242315-2}
and $\bar\lambda = \mathds{1}_{\{\bar y = \psi\}}\max(0, - \Delta \psi)$,
see \cref{the:Qstructure},
we have
\begin{equation*}
	\dual{\bar\eta + \beta \bar \lambda}{y_n - \bar y}
	=
	\dual{\bar\eta + \beta \bar \lambda}{\max(0,y_n - \bar y)}
	\ge
	0
	.
\end{equation*}
Thus, \eqref{eq:draenoidreaoi} implies $\dual{\bar\eta}{S'(\bar u; h)} = 0$.
From the convergence $y_n \to \bar y$ in $L^\infty(\Omega)$, we obtain further that
%, at least after the transition to a subsequence,
it holds $\mathds{1}_{A_{y_n}} \to 0$ pointwise a.e.\ in $\Omega$.
In combination with the boundedness in $L^\infty(\Omega)$ of $\{\mathds{1}_{A_{y_n}}\}$, 
this
implies that
$h_n \mathds{1}_{\Omega \setminus A_{y_n}} \weakly h$
in $L^2(\Omega)$.
Consequently,
\begin{align*}
	0
	&\ge
	\zeta_n + \oo(1)
	\\
	&=
	\frac{1}{2} j''(\bar y) S'(\bar u; h)^2
	+
	\frac\alpha{2+\delta}\,\norm{h_n}_{L^2(\Omega\setminus A_{y_n})}^2
	+
	\frac{\delta\,\alpha}{2\,(2+\delta)} \norm{h_n}_{L^2(\Omega)}^2
	+\oo(1)
	.
\end{align*}
The weak convergence
$h_n \mathds{1}_{\Omega \setminus A_{y_n}} \weakly h$
and $\norm{h_n}_{L^2} = 1$ now
imply
\begin{equation*}
	0
	% \ge
	% \frac{1}{2} j''(\bar y) S'(\bar u; h)^2
	% +
	% \liminf_{n \to \infty}
	% \frac\alpha2\,\norm{h_n}_{L^2(\Omega\setminus A_{y_n})}^2
	% +
	% \frac{\delta\,\alpha}{2\,(2+\delta)} \norm{h_n}_{L^2(A_{y_n})}^2
	\ge
	\frac{1}{2} j''(\bar y) S'(\bar u; h)^2
	+
	\frac\alpha{2+\delta}\,\norm{h}_{L^2(\Omega)}^2
	+
	\frac{\delta\,\alpha}{2\,(2+\delta)}
	.
\end{equation*}
This contradicts \eqref{eq:randomeq2423152} and $\norm{h}_{L^2} \le 1$
and completes the proof of \ref{item:enhancedSSC-i}.

It remains to prove \ref{item:enhancedSSC-ii}. To this end, let us
suppose that $\bar u$ is strongly stationary and satisfies \eqref{eq:globalcondition72623}, \eqref{eq:randomeq827373} and \eqref{eq:globalinequality22}
with some $\bar y$, $\bar \lambda$, $\bar \eta$, $\bar p$, $\bar \nu$, $\mu$ and $\beta$.
Then, \cref{the:Qstructure} again implies that $\bar u = u_{\bar y}$ has to hold
with $u_{\bar y}$ as in \eqref{eq:uydef}. 
Consider now an arbitrary but fixed state $y \in H_0^1(\Omega) \cap H^2(\Omega)$
that is attainable in \eqref{eq:P}, and denote with $u_y \in U_{ad}$ and $\lambda_y$ the 
partially optimal control of $y$ and the associated multiplier in \eqref{eq:lambdaydef}, respectively. 
From \eqref{eq:globalcondition72623}, we have $B_y = \emptyset$.
Using \eqref{eq:aux_in_acht_1}, \eqref{eq:aux_in_acht_3}, \eqref{eq:randomeq827373}
and $\bar\lambda = \mathds{1}_{\{\bar y = \psi\}}\max(0, - \Delta \psi)$,
see \cref{the:Qstructure},
we have
\begin{equation*}
	\dual{\bar\eta}{y - \bar y}
	=
	\dual{\bar\eta}{\max(0,y - \bar y)}
	\ge
	-\beta \, \dual{\bar\lambda}{\max(0,y - \bar y)}
	=
	-\beta \, \dual{\bar\lambda-\lambda_y}{\max(0,y - \bar y)}
	.
\end{equation*}
Now, it follows from \eqref{eq:sadroeiiaoedr} and similarly to the derivation of \cref{th:SSCmajor}\ref{item:MajorSSC-ii}  that 
\begin{equation*}
\begin{aligned}
&J(y, u_y) - J(\bar y, \bar u)
\\
&\quad \geq 
\left \langle  \bar p  , \lambda_y - \bar \lambda \right \rangle +
\left \langle  \bar \eta, y - \bar y \right \rangle + ( \bar \nu, u_y - \bar u)_{L^2} + \frac{\mu }{2}\|y - \bar y\|_{L^2}^2 
+ \frac{\alpha}{2} \|u_y - \bar u\|_{L^2}^2
\\
&\quad\geq
- \beta \left \langle  \bar \lambda - \lambda_y, \max(0, y - \bar y) \right \rangle
+ \frac{\mu }{2}\|y - \bar y\|_{L^2}^2 
+ \frac{\alpha}{2} \norm{u_y - \bar u}_{L^2(\Omega\setminus A_y)}^2
\\
&\quad\geq  \beta \int_{\Omega} |\nabla \max(0, y - \bar y) |^2 \mathrm{d}x
+ \beta \int_{\Omega \setminus \{y = \psi\}} (\bar u - u_y)  \max(0, y - \bar y)  \mathrm{d}x
\\
&\qquad\qquad\qquad\qquad + \frac{\mu }{2}\|y - \bar y\|_{L^2}^2 
+ \frac{\alpha}{2} \norm{u_y - \bar u}_{L^2(\Omega\setminus \{y=\psi\})}^2
\\
&\quad\geq 
\frac{1}{2}\left ( \mu + 2 \beta \omega- \frac{\beta^2 }{ \alpha} \right )  \| \max(0, y - \bar y)\|_{L^2}^2
+
\frac{\mu}{2} \| \min(0, y - \bar y)\|_{L^2}^2.
\end{aligned}
\end{equation*}
The claim now follows immediately from \cref{the:Qstructure}
and the one-to-one correspondence between the states $y$ and the partially optimal controls $u_y$. 
\end{proof}

Note that the assumptions \eqref{eq:randomeq242315-2}, \eqref{eq:randomeq827373} and \eqref{eq:globalinequality22}
in \cref{th:enhancedSSC}
are exactly the same as in \cref{th:SSCmajor}. 
For the conditions \eqref{eq:randomeq242315-1} and \eqref{eq:globalcondition72623}, this is different. 
Consider, for example, the special case $u_a = -\infty$ and $u_b = \infty$.
In this situation, \eqref{eq:strongstationarity_2} implies $\bar p = - \alpha \bar u$
and we may recast  \eqref{eq:globalcondition72623} as
\begin{equation*}
\bar p \not \in \big ( 2 \alpha \Delta \psi , 0 \big ) \text{ a.e.\ in } \{\bar y > \psi,\, \Delta \psi < 0 \}.
\end{equation*}
What is remarkable about the above condition is that, in contrast to \eqref{eq:multiplierassumptions-i} 
and the results in \cite{KunischWachsmuth2012,AliDeckelnickHinze2018}, it states that 
optimality can not only be guaranteed when the absolute value of the negative part $\min(0, \bar p)$ 
of the adjoint state $\bar p$
is sufficiently small in the inactive set $\{\bar y > \psi\}$, but also when $\abs{\min(0, \bar p)}$ is sufficiently large in those parts 
of the domain $\Omega$, where $\psi$ and $\bar y$ satisfy $\bar y > \psi$ and $\Delta \psi < 0$.  
Moreover, the behavior of $\bar p$ in the set $\{\bar y > \psi, \, \Delta \psi \geq 0\}$
is completely irrelevant for the second-order conditions in \cref{th:enhancedSSC}. 
At least to the authors' best knowledge, similar effects have not been documented so far in the literature. 

\section{Counterexamples: Strong Stationarity without Optimality}
\label{sec:7}

We conclude this paper with three counterexamples that put
the results of \cref{sec:3,sec:4,sec:5,sec:6} into perspective and demonstrate which effects can prevent 
a strongly stationary point from being locally optimal.

\subsection{Strict Activity and the Influence of the Multiplier $\boldsymbol{\bar \eta}$}
\label{subsec:counter_strictly_active_1d}
First, we construct a strongly stationary point $\bar u$ with state $\bar y$, multiplier $\bar \lambda$
and a triple $(\bar p, \bar \nu, \bar \eta)$ as in \eqref{eq:strongstationarity}
such that the whole domain $\Omega$ is strictly active and such that $\bar u$
is not a local minimum of \eqref{eq:P}: 
Consider the interval $\Omega := (0,1) \subset \R$ and 
fix a number $r > -1/2$.
Thus, the definition
$\bar\lambda(x) := x^r$
yields $\bar\lambda \in L^2(\Omega)$.
From
\eqref{eq:doublecomplementarity}, we know that,
for strong stationarity to hold with an a.e.-positive $\bar \lambda$,
the control $\bar u$ has to vanish a.e.\ in $\Omega$. Therefore, we define $\bar u := 0$. 
Solving the Poisson problem on $(0,1)$ with Dirichlet boundary conditions
and the right-hand side $\bar u + \bar \lambda  = \bar \lambda$ now yields
\begin{equation*}
	\bar y(x)
	=
	\frac{x - x^{r+2}}{(r+2) \, (r+1)}
	=:
	\psi(x)
	.
\end{equation*}
In order to comply with \eqref{eq:strongstationarity},
we further choose
\begin{equation*}
	u_a := - \infty
	,\quad 
	u_b := \infty 
	,\quad
	\bar p := \bar \nu := 0
	,\quad
	\bar\eta := j'(\bar y)
	,\quad
	\alpha > 0.
\end{equation*}
The objective function $j$ will be specified below. 
Due to $\TT_K(\bar y) \cap \bar\lambda\anni = \{0\}$,
it is easy to see that \eqref{eq:strongstationarity}
is satisfied for the above $\bar u$, $\bar y$, $\bar \lambda$, $\bar p$, $\bar \nu$, $\bar \eta$ and $\psi$,
i.e.,
$\bar u$ is strongly stationary.

To show that $\bar u$ is not necessarily a local minimum of \eqref{eq:P}, 
we consider the perturbed controls
\begin{equation*}
	u_t(x)
	:=
	\begin{cases}
		t^r + x^r & \text{if } x \in (0,t) \\
		0 & \text{else}
	\end{cases},\qquad t > 0. 
\end{equation*}
We first give a lower bound for the associated states
$y_t := S(u_t)$.
To this end, we define
\begin{equation*}
	\hat y_t(x) := \psi(x) + \frac{t^r}{2}\,(t \, x - x^2)
	\ge
	\psi(x)
	=
	\bar y(x)
\end{equation*}
for $x \in [0,t]$.
We are going to show that
$y_t \ge \hat y_t$ on $[0,t]$.
First, we check that $y_t > \psi$ a.e.\ on $(0,t)$.
Indeed, if the measure of the set $A_t := \{y_t = \psi\} \cap (0,t)$
was positive,
then
Stampacchia's lemma together with $y_t, \psi \in H^2(0,1)$
would imply
\begin{equation*}
	t^r + x^r
	=
	u_t(x)
	\le
	u_t(x) + \lambda_t(x)
	=
	-\Delta y_t(x)
	=
	-\Delta \psi(x)
	=
	\bar\lambda(x)
	=
	x^r
\end{equation*}
f.a.a.\ $x \in A_t$, where $\lambda_t := -\Delta y_t - u_t$ denotes the multiplier associated with $u_t$, 
and this would be a contradiction.
Hence,
$\lambda_t = 0$ a.e.\ on $(0,t)$
and this, in turn, gives
$-\Delta y_t = u_t = -\Delta \hat y_t$
on $(0,t)$.
Together with
$y_t(0) = 0 = \hat y_t(0)$
and
$y_t(t) \ge \psi(t) = \hat y_t(t)$,
the comparison principle now yields
the desired inequality 
$y_t \ge \hat y_t$ on $(0,t)$.

Next, we compute
\begin{equation*}
	\norm{u_t}_{L^2}^2
	=
	\int_0^t (t^r + x^r)^2 \, \dx
%	=
%	\int_0^t t^{2\,r} + 2 \, t^r \, x^r + x^{2\,r} \, \dx
	=
	\paren[\Big]{1 + \frac{2}{r+1} + \frac1{2\,r + 1} } \, t^{2\,r + 1}
\end{equation*}
and
\begin{align*}
	\norm{y_t - \bar y}_{L^2}^2
	&\ge
	\int_0^t (\hat y_t - \psi)^2 \dx
	=
	\frac{t^{2\,r}}{4}\,\int_0^t (t \, x - x^2)^2 \, \dx
	\\
	&=
	\frac{t^{2\,r}}{4}\,\int_0^t t^2 \, x^2 - 2 \, t \, x^3 + x^4 \, \dx
	=
	\frac{t^{2\,r + 5}}{4}\,\paren[\Big]{\frac13 - \frac12 + \frac 15}
	=
	\frac{t^{2 \, r + 5}}{120}.
\end{align*}
From now on,
we additionally assume that $r > 3/2$. For this choice of $r$, 
the last two estimates show that, for some constant $c > 0$, 
we have
\begin{equation*}
	\frac{\norm{y_t-\bar y}_{L^2}}{\norm{u_t}_{L^2}^2}
	\ge
	c \, t^{r + 5/2 - 2 \,r - 1}
	=
	c \, t^{3/2 - r}
	\to
	\infty
\end{equation*}
as $t \searrow 0$.
Hence,
the Banach-Steinhaus theorem implies the existence of
$g \in L^2(\Omega)$ with
\begin{equation*}
	\frac{\abs{\scalarprod{y_t-\bar y}{g}_{L^2}}}{\norm{u_t}_{L^2}^2}
	\to
	\infty
	.
\end{equation*}
In fact, due to $r > 3/2$, we can choose $\gamma \in (2-r, 1/2)$
and $g(x) := x^{-\gamma}$.
Indeed, due to the inequality $-\gamma> -1/2$ and the properties of $y_t$ and $\bar y$, we have $g \in L^2(0,1)$
and 
\begin{align*}
	\scalarprod{y_t-\bar y}{g}_{L^2}
	&\ge
	\int_0^t (\hat y_t - \psi) \, g \, \dx
	=
	\frac{t^r}{2} \, \int_0^t (x \, t - x^2) \, x^{-\gamma} \, \dx
	\\
	&=
	\frac{t^r}{2} \, \paren[\Big]{\frac1{2-\gamma} - \frac1{3-\gamma}} \, t^{3 - \gamma}
	=
	\frac{t^{r + 3 - \gamma}}{2 \, (2-\gamma) \, (3-\gamma)}
	.
\end{align*}
There thus exists a constant $c > 0$ such that 
\begin{equation*}
	\frac{\scalarprod{y_t-\bar y}{g}_{L^2}}{\norm{u_t}_{L^2}^2}
	\ge
	c \, t^{r + 3 - \gamma - 2 \, r - 1}
	=
	c \, t^{2 - \gamma - r}
	\to
	\infty.
\end{equation*}
If we now define 
\begin{equation}
\label{eq:randomobjective725362}
j : H_0^1(\Omega) \to \R,\qquad j(y) := -\scalarprod{y}{g}_{L^2},
\end{equation}
then it holds 
\begin{equation*}
	j(y_t) + \frac\alpha2 \, \norm{u_t}_{L^2}^2
	<
	j(\bar y)
	=
	j(\bar y) + \frac\alpha2 \, \norm{\bar u}_{L^2}^2
\end{equation*}
for all small enough $t > 0$. 
Together with $u_t \to 0 = \bar u$ in $L^2(\Omega)$,
this shows that $\bar u$ cannot be a local minimizer of the problem \eqref{eq:P} with $j$ chosen as in 
\eqref{eq:randomobjective725362}. (Note that the function $j$ in \eqref{eq:randomobjective725362} 
technically does not satisfy the conditions in \cref{assumption:standing} since it is not bounded from below. 
This can easily be corrected by redefining $j$ away from $\bar y$. 
We omit this modification here and in the next two subsections for the sake of simplicity.)

The reason for the non-optimality of $\bar u$ in the above example 
is precisely the $\bar \eta$-term in the expansion \eqref{eq:Taylorexpansion}. 
In particular, for $\bar \eta (x) = - g(x) = -x^{-\gamma}$ and $\bar \lambda(x) = x^{r}$ 
with exponents $r > 3/2$ and $\gamma \in (2-r, 1/2)$, we trivially have 
\begin{equation*}
 \bar \eta(x) + \beta \bar \lambda(x) = -x^{-\gamma} + \beta x^r \not \geq 0
\end{equation*}
for every choice of the parameter $\beta \geq 0$
so that, e.g., the condition \eqref{eq:randomeq242315-2} in \cref{th:enhancedSSC} is always violated.
Since the assumptions \eqref{eq:randomeq242315-1} and \eqref{eq:randomeq2423152} 
are obviously satisfied for the control $\bar u = 0$ and the objective \eqref{eq:randomobjective725362},
this demonstrates that the majorizability condition on the multiplier $\bar \eta$ 
in our second-order sufficient optimality conditions is necessary and cannot be dropped.

\subsection{Inactivity and the Influence of the Adjoint State $\boldsymbol{\bar p}$}
\label{subsec:counter_inactive_1d}
Next, we construct a strongly stationary point such that the whole domain $\Omega$
is inactive and such that the $\bar p$-term in \eqref{eq:Taylorexpansion} prevents $\bar u$
from being a local minimum. 
As before, we consider the interval $\Omega := (0,1)$
and the bounds $u_a = -\infty$, $u_b = \infty$.
In order to satisfy the system of strong stationarity \eqref{eq:strongstationarity},
we define
\begin{equation}
\label{eq:ex2p}
	\bar p(x) := - x \, (1-x).
\end{equation}
If we set $\alpha := 1$, then the above choice leads to
\begin{equation}
\label{eq:ex2u}
	\bar u(x) = -\bar p (x) = x \, (1-x)
\end{equation}
and we may solve the Poisson problem $-\Delta \bar y = \bar u$ with homogeneous Dirichlet boundary conditions
to obtain 
\begin{equation}
\label{eq:ex2y}
	\bar y(x) = \frac{x^4}{12} - \frac{x^3}{6} + \frac{x}{12}.
\end{equation}
To achieve inactivity (almost) everywhere in $\Omega$, we further set
\begin{equation}
\label{eq:ex2psi}
\psi(x) := \bar y(x) - c \, x^2 
\end{equation}
for some arbitrary but fixed $c \in (0,1/8)$, 
and to comply with \eqref{eq:strongstationarity_1}, we define the
state-dependent part of the objective function via
\begin{equation*}
	j(y) := \left \langle -\Delta \bar p, y \right \rangle = - 2 \int_\Omega y\, \mathrm{d}x. 
\end{equation*}
Now, it is straightforward to check that
\eqref{eq:strongstationarity} is satisfied
with
$\bar \nu = \bar\lambda = \bar\eta = 0$.

It remains to check that $(\bar y, \bar u)$
is not a local solution of \eqref{eq:P}.
To this end, we define the modified controls
\begin{equation*}
	u_t(x)
	:=
	\begin{cases}
		0 & \text{for } x \in (0,t), \\
		\bar u(x) +  2 \, c \, \frac{t^2 - 2\,t}{(1-t)^2} & \text{for } x \in (t,1),
	\end{cases} \qquad t \in (0, 1).
\end{equation*}
We claim that the states $y_t := S(u_t)$
associated with the above $u_t$, $t \in (0, 1)$, are precisely the functions
\begin{equation*}
	y_t(x)
	=
	\begin{cases}
		\psi(x) & \text{for } x \in (0,t) \\
		\bar y(x) + c (1-x) \, \left ( \frac{t^2 - 2\,t}{(1-t)^2} \, x +  \frac{t^2}{(1-t)^2}  \right ) & \text{for } x \in (t,1)
	\end{cases}.
\end{equation*}
Indeed, a direct calculation shows 
\begin{equation*}
y_t(t) = \bar y(t) - c  t^2 = \psi(t),
\qquad y_t'(t) = \bar y'(t) - 2 ct = \psi'(t),
\qquad y_t(0) = y_t(1) = 0,
\end{equation*}
so that $y_t$ is an element of $H_0^1(\Omega) \cap H^2(\Omega )$, 
and by exploiting this $H^2$-regularity, it is easy to check that 
\begin{equation*}
y_t(x) - \psi(x) 
\geq 
\begin{cases}
0 & \text{ for all  } x \in (0, t)
\\
 c x^2 + c (1-x) \, \left ( \frac{t^2 - 2\,t}{(1-t)^2} \, x +  \frac{t^2}{(1-t)^2}  \right ) \geq 0 & \text{ for all  } x \in (t, 1)
\end{cases}
\end{equation*}
and
\begin{equation*}
\lambda_t := - \Delta y_t - u_t =
\begin{cases}
 - \Delta \psi \geq 0 & \text{a.e.\ in } (0, t)
\\
0 & \text{a.e.\ in } (t, 1)
\end{cases}.
\end{equation*}
Thus, $y_t = S(u_t)$ as desired. Using \cref{lemma:Taylorlikeexpansion}, we may now compute that 
\begin{equation}
\label{eq:randomeq1635}
\begin{aligned}
J(y_t, u_t) - J(\bar y, \bar u)
&= 
\left \langle  \bar p  , \lambda_t \right \rangle + \frac{1}{2} \|u_t - \bar u\|_{L^2}^2
\\
&= \int_0^t \bar u \, \Delta \psi \, \dx + \frac12 \, \left ( \int_0^t \bar u^2 \dx 
+ \int_t^1 \left ( 2 \, c \, \frac{t^2 - 2\,t}{(1-t)^2} \right )^2
 \, \dx \right )
\\
&= \int_0^t - \frac12 \bar u^2 - 2c \bar u\dx + 2 \, c^2 \, \frac{(t^2 - 2\,t)^2}{(1-t)^3} 
\\
&= \int_0^t - \frac12 (x^2 - x)^2 - 2c (x - x^2)\dx + 2 \, c^2 \, \frac{(t^2 - 2\,t)^2}{(1-t)^3} 
\\
&= (- c + 8 c^2) t^2 + \oo(t^2),
\end{aligned}
\end{equation}
where the Landau symbol refers to the limit $t \searrow 0$. Since $c$ was chosen to be an element of the interval $(0, 1/8)$,
\eqref{eq:randomeq1635} implies that $\bar u$ is indeed not a local minimizer. 

Note that, for the above $\bar p$, $\bar \eta$, $\bar \nu $, $\alpha$ and  $j$, the expansion \eqref{eq:Taylorexpansion} yields
\begin{equation*}
\begin{aligned}
 J(y, u) - J(\bar y, \bar u)
&  =
\left \langle  \bar p  , \lambda - \bar \lambda \right \rangle +  \frac{1}{2} \|u - \bar u\|_{L^2}^2.
\end{aligned}
\end{equation*} 
The term in \eqref{eq:Taylorexpansion} that is responsible for the behavior in \eqref{eq:randomeq1635} is thus precisely the one
which involves the adjoint state $\bar p$. It is further easy to check that 
the functions $\bar p$, $\bar u$, $\bar y$, and $\psi$ in \eqref{eq:ex2p}, \eqref{eq:ex2u}, \eqref{eq:ex2y}
and \eqref{eq:ex2psi} satisfy 
\begin{equation*}
0 \geq \frac{ -\alpha \bar u}{ \bar y - \psi} = \frac{ \bar p}{ \bar y - \psi} = \frac{ x - 1}{ c x } \to -\infty \text{ for } x \to 0,
\end{equation*}
and 
\begin{equation*}
-2 \Delta \psi =  2 \left ( -\Delta \bar y + 2c  \right )
=   2\left (  \bar u  + 2c \right )\geq 2 \bar u \geq \bar u \geq 0 \quad \text{ a.e.\ in }\Omega.
\end{equation*}
This shows that the conditions 
\begin{equation*}
\bar p + \beta (\bar y - \psi)  \geq 0  \text{ a.e.\ in } \{0 < \bar y - \psi < \gamma\}
\end{equation*}
and 
\begin{equation*}
	-\alpha \bar u + \beta (\bar y - \psi) \geq 0 \text{ a.e.\ in } \{0 < \bar y - \psi < \gamma, \,\Delta \psi < 0, \, 0 < \bar u < -(2+\delta)\Delta \psi\},
\end{equation*}
in \cref{th:SSCmajor,th:enhancedSSC} are violated 
for every choice of the parameters $\beta \geq 0$, $\gamma > 0$ and $\delta > 0$, and, since \eqref{eq:randomeq242315-2} and 
\eqref{eq:randomeq2423152} trivially hold for $j'' = 0$ and $\bar \eta = 0$, 
that additional assumptions on $\bar u$ and $\bar p$ 
(or additional curvature terms in \eqref{eq:randomeq2423152} involving these quantities, respectively) are necessary for 
a second-order condition to hold in the above situation. 

\subsection{Non-Negligibility of Sets with Zero Capacity}
\label{subsec:counter_inactive_2d}

In what follows, we demonstrate by means of a final example that the contact set $\{\bar y= \psi \}$
is relevant for the derivation of second-order optimality conditions for problems of the type \eqref{eq:P}
even if it has $H^1$-capacity zero and is thus negligible in the first-order conditions
\eqref{eq:Bouligandstatcond} and \eqref{eq:strongstationarity}. 
Let us denote with $U_t(0)$, $t>0$, the open ball of radius $t$ around the origin
in the two-dimensional Euclidean space and define $\Omega := U_1(0) \subset \R^2$.
Since all functions in the following counterexample will be rotationally symmetric,
it is convenient to work with the Laplacian in polar coordinates, i.e.,
\begin{equation*}
	\Delta f = \frac1r \, \frac{\partial}{\partial r} \, \paren[\Big]{r \, \frac{\partial}{\partial r} f}
\end{equation*}
for rotationally symmetric $f$.
As before, 
we choose $\alpha := 1$, $u_a := - \infty$ and $u_b := \infty$. 

To construct a point which satisfies the strong stationarity system \eqref{eq:strongstationarity},
we define the adjoint via
$\bar p(r) := r^2 - 1$.
This leads to
\begin{equation*}
	\bar u(r) = -\bar p(r) = 1 - r^2
	\quad\text{and}\quad \bar \nu = 0.
\end{equation*}
Next, we solve the Poisson equation $-\Delta \bar y = \bar u$
to obtain
\begin{equation}
\label{eq:randomsolution263535}
	\bar y(r)
	=
	\frac1{16}\,r^4 - \frac14 \, r^2 + \frac3{16}
\end{equation}
and define
\begin{equation}
\label{eq:randomeq7253632}
	\psi(r) := \bar y(r) - c \, r^2
\end{equation}
with a constant $c > 0$ (to be fixed below). Due to the identity $-\Delta \bar y = \bar u$ and \eqref{eq:randomeq7253632}, 
it is obvious that $\bar y$ is precisely the solution of the obstacle problem on $\Omega$ with right-hand side $\bar u$ and obstacle $\psi$.
From the properties of $\bar y$ and $\psi$, it follows further that the constraint 
$y \geq \psi$ is only active in the origin in the above situation, 
i.e., the set $\{\bar y = \psi\}$ has $H^1$-capacity zero. Note that 
this implies in particular that $\bar \lambda = 0$ holds and 
that the control-to-state map $S : L^2(\Omega) \to H_0^1(\Omega)$, $u \mapsto y$, associated with \eqref{eq:obstacleproblem}
 is Gâteaux differentiable in $\bar u$, cf.\  \cref{th:solutionmapproperties}.
By choosing
\begin{equation}
\label{eq:randomobjective2635}
	j(y)
	=
	-4 \, \int_\Omega y \, \dx,
\end{equation}
we now obtain that the system \eqref{eq:strongstationarity} is satisfied with $\bar \eta = 0$.
Hence,
the point $(\bar y, \bar u)$ 
is strongly stationary for the problem \eqref{eq:P}.

To prove that $\bar u$ is nonetheless not  a local solution of \eqref{eq:P}, we use an argumentation 
that is similar to that in \cref{subsec:counter_inactive_1d}.
Define
\begin{equation*}
	u_t(r)
	:=
	\begin{cases}
		0 & \text{if } r \in (0,t), \\
		1-r^2 & \text{if } r \in [t,1),
	\end{cases}\qquad t > 0. 
\end{equation*}
Then, $u_t$ trivially satisfies $u_t \to u$ in $L^2(\Omega)$ for $t \searrow 0$,
the states $y_t := S(u_t)$ are clearly rotationally symmetric, and we may use the
comparison principle in \cref{lem:comparison} to deduce that $y_t \le \bar y$ holds a.e.\ in $\Omega$.
To obtain a reverse estimate, 
we consider the value $y_t(t)$, i.e., the value of $y_t$ at the radius $r = t$.
From $\psi \le y_t \le \bar y$,
it follows that $\abs{y_t(t) - \bar y(t)} \le c \, t^2$.
We claim that we even have $\norm{ y_t - \bar y}_{L^\infty} \le c \, t^2$.
On the inner ball $U_t(0)$, this inequality is obvious since
$0 \leq \bar y - \psi \le c \, t^2$ holds a.e.\ in $U_t(0)$.
Further, on the annulus $U_1(0) \setminus U_t(0)$,
the function $y_t - \bar y \le 0$ is superharmonic.
Thus, it attains its minimum on the boundary, and the desired estimate follows immediately.  

It remains to compare the values of the objective function in \eqref{eq:randomobjective2635}.
For the states, we have
\begin{equation*}
	\abs{ j(y_t) - j(\bar y) }
	\le
	4 \, \int_{\Omega} \abs{y_t - \bar y} \, \dx
	\le
	4 \, \pi \, c \, t^2
	.
\end{equation*}
Further, for the controls, we get
\begin{equation*}
	\frac12 \, \norm{u_t}_{L^2}^2
	-
	\frac12 \, \norm{\bar u}_{L^2}^2
	=
	-
	\pi
	\,
	\int_0^t (1 - r^2)^2 \, r \, \d r
	=
	\left (-\frac{t^2}{2} + \frac{t^4}2 - \frac{t^6}{6} \right )\pi
	.
\end{equation*}
Hence,
\begin{equation*}
	J(y_t, u_t) - J(\bar y, \bar u)
	\le
	 \left (4 c t^2 -\frac{t^2}{2} + \frac{t^4}2 - \frac{t^6}{6} \right )\pi
	.
\end{equation*}
The right-hand side of this inequality is negative for
$c < 1/8$
and $t > 0$ small enough.
This shows that the strongly stationary point $(\bar y, \bar u)$
cannot be a local minimizer for \eqref{eq:P}
although the objective function  is linear in $y$ and strongly convex in $u$.

The reason for the non-optimality of the tuple $(\bar y, \bar u)$ in the above example 
is essentially the same as in \cref{subsec:counter_inactive_1d}. Due to the properties 
of the adjoint state $\bar p$, the state $\bar y$ and the obstacle $\psi$, the $\bar p$-term in 
\eqref{eq:Taylorexpansion} becomes negative and goes to zero too slowly in the limit $u \to \bar u$
to be compensated by the quadratic expression $\frac{\alpha}{2} \|u - \bar u\|_{L^2}^2$.
What is remarkable in the situation of \eqref{eq:randomsolution263535} is that 
this effect is present although the contact set $\{\bar y = \psi \}$ has $H^1$-capacity zero 
and is thus completely irrelevant in the first-order optimality conditions \eqref{eq:Bouligandstatcond} and \eqref{eq:strongstationarity}.
To be more precise, we can  observe here that the sequence
\begin{equation*}
 \frac{1}{\|u - \bar u\|_{L^2}^2} (\lambda - \bar \lambda) \mathds{1}_{\{\bar y > \psi\}} \in L^2(\Omega)
\end{equation*}
appearing, e.g., in the proofs of \cref{th:SSCmajor,th:enhancedSSC}
exhibits a singular limiting behavior for $u \to \bar u$ and that the expression 
$
\|u - \bar u\|_{L^2}^{-2} \left \langle  \bar p  , \lambda  - \bar \lambda \right \rangle$
in the expansions \eqref{eq:randomestimate434} and \eqref{eq:randomeq2816363} tends to a singular term 
that depends on the function value of the negative part of the adjoint state $\bar p$ at the origin. 
Note that a similar behavior cannot occur in the one-dimensional setting where \eqref{eq:strongstationarity_3}
necessarily implies $\min(0, \bar p) = 0$ everywhere on $\{\bar y = \psi\}$.
The above considerations indicate that the constraint $S(u) = y$ in the 
optimal control problem \eqref{eq:P} induces additional curvature effects that 
depend on the fine properties of the adjoint state $\bar p$, the state $\bar y$
and the obstacle $\psi$. 
Note that similar observations have also been made in 
the context of bang-bang optimal control problems in 
\cite{ChristofWachsmuth2018}, 
the sensitivity analysis of elliptic variational inequalities of the second kind in \cite{ChristofH012018}, and 
necessary optimality conditions for state-constrained problems in \cite{NhuSonYao2016}. 
We leave a detailed analysis of the emerging distributional curvature terms for future work.

%%fakesection: Bibliography
\ifbiber
	% We are cheating a little bit ;)
	\renewcommand*{\bibfont}{\small}
	\printbibliography
\else
	\bibliographystyle{plainnat}
	\bibliography{references}
\fi
\end{document}